\documentclass[11pt]{article}

\usepackage{amsmath,amssymb,amsthm} 
\usepackage{listings,verbatim,float}
\usepackage{graphicx}
\usepackage{tikz}
\usepackage{geometry}
\usepackage{forest}
\usepackage{ragged2e} 
\usepackage{setspace} 
\usepackage[style=authoryear,backend=biber]{biblatex}
\usepackage{hyperref} 
\usepackage{xurl}
\usepackage{bm}
\usepackage{booktabs}
\usepackage{fancyhdr}

\newcommand{\thetitle}{}
\newcommand{\theauthor}{}
\renewcommand{\title}[1]{\renewcommand{\thetitle}{#1}}
\renewcommand{\author}[1]{\renewcommand{\theauthor}{#1}}
\hypersetup{
 colorlinks=true,
 linkcolor=blue,
 citecolor=blue,
 urlcolor=blue,
 pdfauthor={\theauthor},
 pdftitle={\thetitle}
}
\AtBeginDocument{

 \begin{center}
 {\LARGE\textbf{\thetitle}\par}
 \vskip 1.7em
 {\large\theauthor}
 \vskip 2em
 \end{center}
}

\newtheorem{theorem}{Theorem}[section]
\newtheorem{corollary}{Corollary}[theorem]
\newtheorem{lemma}[theorem]{Lemma} 
\newtheorem{prop}[theorem]{Proposition}

\theoremstyle{definition}
\newtheorem{definition}[theorem]{Definition}

\newtheorem{deflem}[theorem]{Definition/Lemma}
\newtheorem{remark}[theorem]{Remark}

\newenvironment{contact}{
 \vspace{0.5em}
 \begin{center}
 \begin{tabular}{@{}l}
}{
 \end{tabular}
 \end{center}
 \vspace{0.5em}
}

\newcommand{\email}[1]{Email:~\url{#1}}

\newtheorem{innercustomthm}{Theorem}
\newtheorem{innercustomlem}{Lemma}

\theoremstyle{definition}
\newtheorem{innercustomdef}{Definition}

\newenvironment{customthm}[1]
  {\renewcommand\theinnercustomthm{#1}\innercustomthm}
  {\endinnercustomthm}

\newenvironment{customlemma}[1]
  {\renewcommand\theinnercustomlem{#1}\innercustomlem}
  {\endinnercustomlem}

\newenvironment{customdef}[1]
  {\renewcommand\theinnercustomdef{#1}\innercustomdef}
  {\endinnercustomdef}

\renewenvironment{abstract}{
 \begin{center}
 {\small\textbf{Abstract}\par}  
 \begin{minipage}{.85\textwidth}
}{
 \end{minipage}
 \end{center}
 \vspace{1em}
}

\title{Probability of Quota Violations in Divisor Apportionment Methods with Nonzero Allocations}
\author{Tyler C. Wunder and Joseph W. Cutrone}

\begin{document}


\begin{abstract}
Apportionment assigns indivisible items among groups. By the Balinski-Young theorem, no method can satisfy both house monotonicity and the quota rule. This paper investigates quota violations caused by nonzero allocation constraints, and derives exact probability formulas for their frequency. Such violations occur in systems like the U.S. House of Representatives, where each state is guaranteed at least one seat. We analyze the three-state case, introduce the $\tau$ statistic to parametrize population distributions, and prove an Asymptotic Quota Stabilization theorem: for fixed $\tau$, quota behavior stabilizes as populations grow, yielding probability results for quota violations determined by the set of ultimately violatory $\tau$ values.\\

Applying this framework to the five classical divisor methods, we derive exact probability formulas. Additionally, we show that as the number of seats $M \to \infty$, these probabilities converge to method-specific constants. These results provide a precise, quantitative foundation for evaluating the fairness and frequency of quota violations in constrained apportionment systems. 

\end{abstract}

\section{Introduction and Foundations}

\subsection{Background on Apportionment}
Let $n$ be the number of states with populations $p_i \in \mathbb{R}_{>0}$ for $1\leq i \leq n$. To avoid ties, assume no two states have the same populations. Let $M \geq n$ be the number of seats to apportion. 

A \textit{divisor method} is an algorithmic procedure to distribute the $M$ seats by dividing populations by a strictly increasing \textit{divisor function}, $\delta: \mathbb{R}_{\geq 0} \to \mathbb{R}_{\geq 0}$.

Given populations $(p_1, \dots, p_n)$ and a divisor function $\delta$, an \textit{apportionment} $A(p_1, \dots, p_n) = (a_1, \dots, a_n)$, is calculated as follows:
\begin{enumerate}
\item[\emph{1.}] Assign each state zero seats (or one seat if $\delta(0) = 0$).
\item[\emph{2.}] Calculate each state's priority value $\frac{p_i}{\delta(r_i)}$, where $r_i$ is the current number of seats assigned to state $i$. 
\item[\emph{3.}] Assign to the state with the highest priority value an additional seat.
\item[\emph{4.}] Repeat steps (2) and (3) until $M$ seats have been apportioned.
\end{enumerate}

The five ``workable methods" most commonly used are defined by the following divisor functions:
\begin{center}
\begin{tabular}{|l|l|}
 \hline
 Method & Divisor Function \\ \hline
 Adams & $\delta(s) = s$ \\
 Jefferson &$\delta(s) = s+1$\\
 Webster & $\delta(s) = s+0.5$\\
 Huntington-Hill & $\delta(s) = \sqrt{s(s+1)}$ \\
 Dean & $\delta(s) = \dfrac{2s(s+1)}{2s+1}$ \\ \hline
\end{tabular}
\end{center}

A state's \textit{standard quota} $q_i$ is the number of seats proportional to its population. This generally non-integer number is defined by $q_i = \frac{p_i}{SD}$, where $SD = \frac{P}{M}$ is the \textit{standard divisor}, defined to be the total population $P = \sum_{i=1}^n p_i$ divided by the number of seats $M$. Ideally, each state should be apportioned its standard quota. However, as the standard quota is rarely an integer, a state should instead receive its \emph{upper quota}, $\lceil q_i \rceil$, or its \emph{lower quota}, $\lfloor q_i \rfloor$. With this in mind, define the following notion of unfair apportionment:

\label{q.v._def} A \emph{quota violation} occurs whenever a state is assigned more seats than its upper quota or fewer seats than its lower quota. That is, $a_i > \lceil q_i \rceil$ or $a_i < \lfloor q_i \rfloor$ for some state $i$. In particular, a \emph{lower quota violation} occurs when $a_i < \lfloor q_i \rfloor$, and an \emph{upper quota violation} occurs when $a_i > \lceil q_i \rceil$. 

By the \emph{Balinski-Young Impossibility Theorem} \parencite{b-young_textbook}, all apportionment methods have either quota violations or other unfair occurrences called \textit{paradoxes}. Specifically, the previously defined divisor methods are without paradoxes but all are known to violate quota.

\subsection{Nonzero Allocation Constraints}
In many applications, such as the apportionment of legislative representatives, it is undesirable for an apportionment method to assign zero representatives to any state. To ensure $a_i \geq 1$, methods like Jefferson and Webster are modified such that $\delta(0) = 0$. This modification can introduce additional quota violations not present in the unconstrained method where a small state receives a seat it would not have earned based purely on its priority value.

\begin{definition}\label{stealing_def}
For ordered populations $p_1, \dots, p_n$; $M$ seats; and a divisor method $A$ that guarantees nonzero allocations, define the \emph{modified method} $\tilde{A}$ as follows: \end{definition}
\begin{enumerate}

\item With $M-1$ seats, calculate $A(p_{2}, \dots, p_n)$ to obtain $(a_{2}, \dots, a_n)$.
\item Set $a_1 =1$.
\item Define $\tilde{A}(p_1, \dots, p_n) = (a_1, \dots, a_n)$
\end{enumerate}

Note the order of the steps is important here. The modified method enforces a guaranteed seat for state one and applies the divisor method on the remaining $M-1$ seats. 

\begin{definition}\label{vio_as_not0allo_def}
Let $A$ be a divisor method that guarantees nonzero allocations on $M$ seats and suppose the populations are ordered such that $p_1 < \cdots < p_n$. The apportionment $A(p_1, \dots, p_n)$ exhibits a \emph{quota violation caused by nonzero allocations} if the following conditions hold:
\begin{enumerate}
\item $A(p_1, \dots, p_n)$ violates the quota rule.
\item $\tilde{A}(p_1, \dots, p_n) = A(p_1, \dots, p_n)$.
\item When apportioning $M$ seats $A(p_{2}, \dots, p_n)$ does not violate the quota rule.
\end{enumerate}
\end{definition}

Additionally, define a \emph{quota violation caused or worsened by nonzero allocations} as when conditions 1 and 2 hold.

The motivation behind these definitions is to check if the smallest state has a priority value so low that they would not receive seats by having the largest priority value, and instead only receive a seat because the method guarantees each state at least one seat.

\section{Main Results}
The paper introduces a statistic $\tau$ which is useful for studying quota violations caused by nonzero allocations in the three state case.

\begin{customdef}{\ref{tau_def}}
    Let $(1,x,y)$, $1<x<y$, be the associated reduced population vector for a three state distribution ($p_1, p_2, p_3$). Define the $\tau$ statistic as follows: 
$$ \tau := \frac{\mathrm{mean}(1,x,y)-\mathrm{median}(1,x,y)}{\max(1,x,y)}. $$
\end{customdef}

\begin{customlemma}{\ref{tau_family}}
    The statistic $\tau$ parametrizes a family of population distributions dependent on $x$. Specifically, this is the set of reduced population vectors associated with a fixed $\tau$: $\{ (1,x,y(x,\tau)) \mid x \in \mathbb{R}, x>1 \}$, where $y$ is a linear function of $x$.
\end{customlemma}

\begin{customthm}{\ref{tau_limit_thm1}}
    (Asymptotic Apportionment) Let $(1,x,y(x,\tau))$ be a reduced population vector associated with a fixed $\tau$, and let $A$ be a divisor method that guarantees nonzero allocations on $M$ seats. Then, for all but finitely many $\tau \in (-\frac13,\frac13)$: $$\displaystyle{\lim_{x \to \infty} A(1,x,y(x,\tau)) = \tilde{A} \left( \frac{M(1-3\tau)}{3-3\tau},\frac{2M}{3-3\tau} \right)},$$ where $\tilde{A}$ is defined as apportioning the remaining $M-1$ seats among states 2 and 3 according to the divisor method $A$ (see definition \ref{stealing_def}).
\end{customthm}

\begin{customthm}{\ref{X_tau_def/exist}} (Asymptotic Quota Stabilization for Nonexceptional $\tau$) Let $\mathcal{E}$ be the exceptional set of finite $\tau \in (-\frac 13,\frac 13)$ such that Theorem \ref{tau_limit_thm1} fails. Fix $\tau \notin \mathcal{E}$. Then there exists some $x_\tau$ such that for all $x \geq x_\tau$, $A(1,x,y(x,\tau))$ is always a lower quota violation caused by nonzero allocation or never a violation, with this behavior determined by $\tau$. 

Specifically, let state $i \in \{2,3\}$ be the state that receives the last seat when apportioning $M$ seats among only states $2$ and $3$ from $A(\tilde{q}_2, \tilde{q}_3)$. 
Then: 
\begin{enumerate}
 \item If $A(\tilde{q}_2, \tilde{q}_3)$ assigns state $i$ its lower quota 
 $\lfloor \tilde{q}_i \rfloor$, then for all $x \geq x_\tau$, the apportionment 
 $A(1, x, y(x,\tau))$ is a lower quota violation caused by nonzero allocations, 
 with state $i$ receiving $\lfloor \tilde{q}_i \rfloor - 1$ seats.
 \item If $A(\tilde{q}_2, \tilde{q}_3)$ assigns state $i$ its upper quota 
 $\lceil \tilde{q}_i \rceil$, then for all $x \geq x_\tau$, the apportionment 
 $A(1, x, y(x,\tau))$ is not a quota violation.
\end{enumerate}
\end{customthm}

\begin{customthm}{\ref{prob_sum}} 
 Let $A$ be a divisor method with divisor function $\delta$ that guarantees nonzero allocations on M seats and $n = 3$ states. Then, under the asymptotic uniform distribution on reduced population vectors $(1, x, y)$, with $1<x<y$,  the probability of a quota violation is the same as the probability of a quota violation caused by states guaranteed nonzero allocations and is calculated by:
$$P(A(1,x,y) \text { has a q.v.}) = \sum_{k = \lceil M/2 \rceil}^{M-1} M\left( \frac{1}{k} - \frac{1}{k + D(k)} \right)$$
where $k$ indexes the integer values of $\lfloor \tilde{q}_3 \rfloor$, where 
$$\tilde{q}_3 = \frac{2M}{3-3\tau}, \qquad \tau = \frac{\mathrm{mean}(1,x,y)-\mathrm{median}(1,x,y)}{\max(1,x,y)}=\dfrac{1 - 2x + y}{3y}.$$ and 
$$D(k) =\frac{(M-k)\delta(k-1)-k\,\delta(M-k-1)}{\delta(k-1)+\delta(M-k-1)}.$$
\end{customthm}

\begin{customthm}{\ref{Mtoinftyproblimit}}
    Under the asymptotic uniform distribution for vectors $(1,x,y)$ with $1<x<y$, we have as $M \to \infty$ that the probability of a quota violation caused by nonzero allocations converges to $2 \ln(2) - 1 \approx 0.386292$ in the Adams method, to $0$ in the modified Jefferson method, and to $\ln(2) - \frac{1}{2} \approx 0.193147$ in the Huntington-Hill, Webster, and Dean methods.
\end{customthm}

\begin{customthm}{\ref{kitchen_sink}}
Let $M$ be the total number of seats, and let $A$ be the Modified Jefferson, Modified Webster, Dean, or Huntington-Hill method with divisor function $\delta$. Let $(1,x,y)$ be a reduced population vector with $1 < x < y$, drawn from a probability distribution with joint density $f(x,y)$.

For each $\tau \in \left(-\frac{1}{3},\frac{1}{3}\right)$, the $\tau$-line parametrization in $x$ is 
$$
x = x(y,\tau) := \frac{1}{2}y - \frac{3}{2}\tau y + \frac{1}{2},
$$
Let the set of ultimately violatory $\tau$ values be
$$
V = \bigcup_{k= \lfloor M/2 \rfloor + 1}^{M-1} 
\left( 1 - \frac{2M}{3k}, \, 1 - \frac{2M}{3\bigl(k + D(k)\bigr)} \right),
$$
where
$$
D(k) =\frac{(M-k)\delta(k-1)-k\,\delta(M-k-1)}{\delta(k-1)+\delta(M-k-1)}.
$$

For each $\tau$, let $y_\tau^F$ denote the floor stabilization threshold and $y_\tau$ the agreement with asymptotic apportionment threshold defined in Lemmas \ref{y_tau.def} and \ref{y_tau.def2}. Define
$$
y_\tau^{\max} :=
\begin{cases}
\infty & \text{if } \tau \in V,\\
y_\tau & \text{if } \tau \notin V.
\end{cases}
$$

Then the probability of a quota violation caused by a guaranteed seat is
$$
\Pr(\text{quota violation})
=
\int_{-1/3}^{1/3}
\int_{y_\tau^F}^{y_\tau^{\max}}
f\bigl(x(y,\tau),\,y\bigr)\,
\frac{3}{2}y
\, dy \, d\tau.
$$
\end{customthm}

\section{Assumptions and Background Information}
We assume $M \geq n$ and no ties in priority values. The following are simple lemmas and known results:
\begin{lemma}\parencite{b-young_textbook}\label{neutrality} Divisor methods satisfy neutrality and proportionality: they are invariant under permutation of the states and under positive scaling of all populations. \end{lemma}

Therefore, we may analyze just the standard quotas $  (q_1, \dots, q_n)$ or the
\emph{reduced population vectors} $\left(1,\frac{p_2}{p_1}, \dots, \frac{p_n}{p_1} \right)$.

\begin{theorem}\label{no_up_and_down}\parencite{b-young_textbook} No single divisor-method apportionment can simultaneously contain both an upper and a lower quota violation. \end{theorem}

\begin{theorem}\label{no_vio_n=2}\parencite{b-young_textbook} When $n=2$, no quota violations occur in any divisor method. \end{theorem}

This paper builds off of the authors' prior work. A main result is listed here for convenience which will be used later in the paper. 

\begin{theorem}\label{q.v_criteria_test} \parencite{Self} \textit{(Lower Quota Violation Criteria Test)} Let $A$ be Adams's, Dean's, or the Huntington-Hill method, with divisor function $\delta(s)$. Order the populations such that $p_1 <p_2 < p_3$ and calculate the standard quotas $q_i$. Then the apportionment $A(p_1, p_2, p_3)$ on $M$ seats has a lower quota violation if and only if all three of the following statements are true: 
\begin{enumerate}
\item $q_3 \delta(\lfloor q_1 \rfloor) < q_1 \delta( \lfloor q_3 \rfloor - 1)$
\item $q_3 \delta(\lfloor q_2 \rfloor) < q_2 \delta( \lfloor q_3 \rfloor - 1)$
\item $\lceil q_1 \rceil + \lceil q_2 \rceil + \lfloor q_3 \rfloor = M+1 $ (equivalently for $q_1, q_2 \notin \mathbb{Z}$, $\lfloor q_1 \rfloor + \lfloor q_2 \rfloor + \lfloor q_3 \rfloor = M-1) $.
\end{enumerate}
\end{theorem}

\begin{remark}
By repeating the arguments used in \parencite{Self}, the above test still holds for modified Webster's and modified Jefferson's methods.
\end{remark}

An immediate extension of Theorem 5.5 in \parencite{Self} is:

\begin{theorem}\label{36}
Let $A$ be a divisor method with divisor function $\delta$ satisfying $\delta(0) = 0$ and $\dfrac{x+1}{\delta(\lfloor x \rfloor)}$ is a decreasing function for $x \in [1, \infty)$. Then, if $A(q_1, q_2, q_3)$ produces a quota violation with $q_1<q_2<q_3$, $A(q_1, q_2, q_3) =(\lceil q_1 \rceil, \lceil q_2 \rceil, \lfloor q_3 \rfloor -1)$. In particular, this holds for Huntington-Hill's, Adams', Dean's, modified Jefferson's, and modified Webster's methods.
\end{theorem}

\section{Asymptotic Quota Stabilization via the $\tau$-Statistic}

This section introduces a new statistic, $\tau$, which is useful for studying quota violations and quota violations caused by nonzero allocations. For simplicity, the section begins by discussing three states, and later explains how these results extend to $n$ states.

\subsection{Defining Relative Skewness for $n=3$ States}

\begin{definition}\label{tau_def} Let $(1,x,y)$, $1<x<y$, be the associated reduced population vector for a three state distribution ($p_1, p_2, p_3$). Define the $\tau$ statistic as follows: 
$$ \tau := \frac{\mathrm{mean}(1,x,y)-\mathrm{median}(1,x,y)}{\max(1,x,y)}. $$
\end{definition}
\begin{lemma}\label{tau=proportional}
The statistic $\tau$ is preserved under scaling: $\tau (1,x,y) = \tau (\lambda , \lambda x, \lambda y)$ for $\lambda \in \mathbb{R}_{>0}$. As such, $\tau$ can be calculated for any three state distribution, not just reduced population vectors.
\end{lemma}

\begin{proof}
$$\tau(\lambda,\lambda x,\lambda y)= \frac{\frac{\lambda + \lambda x+ \lambda y}{3} - \lambda x}{\lambda y} =\frac{ \frac{1+x+y}{3} - x}{y} =\tau(1,x,y). $$
\end{proof}

\begin{lemma}\label{tau_family} The statistic $\tau$ parametrizes a family of population distributions dependent on $x$. Specifically, this is the set of reduced population vectors associated with a fixed $\tau$: $\{ (1,x,y(x,\tau)) \mid x \in \mathbb{R}, x>1 \}$, where $y$ is a linear function of $x$. \end{lemma}
\begin{proof} 
Fix $\tau$. Then, for some $(1,x,y)$ with $1<x<y$, by definition
$$ \tau= \frac{ \frac{1+x+y}{3} - x}{y} $$
Solving for $y$ gives
$$ y(x,\tau)= \frac{2x-1}{1-3\tau} = \frac{2}{1-3\tau}x - \frac{1}{1-3\tau}. $$
Thus, for a fixed $\tau$, $y$ is a linear function of $x$ and  the set of reduced population vectors forms a one-dimensional family in $\{(x, y)\, |\, 1<x<y\}$.
\end{proof}

\noindent The last lemma allows us to not think of reduced population vectors $(1,x,y)$ as independent triplets but instead as parametrized rays, with a linear relationship between $x$ and $y$, in the space of allowable reduced population vectors on the $xy$-plane.

\begin{figure}[H]
 \centering
 \includegraphics[width=0.66\linewidth]{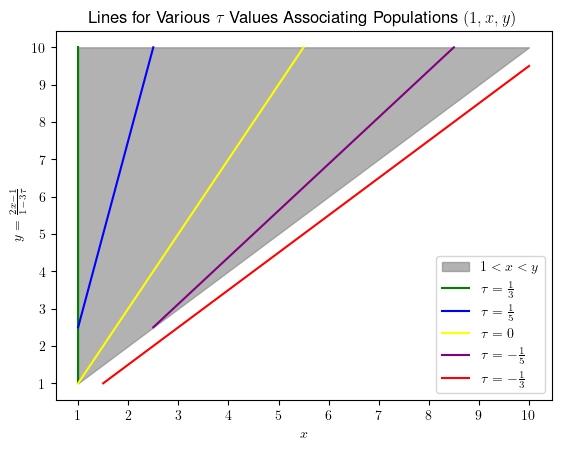}
 \caption{Geometric interpretation of $\tau$}
\end{figure}

\subsection{Properties and Bounds}
\begin{lemma}\label{t_thirds} For any three state distribution with reduced population $(1,x,y)$, with $1 < x < y$, the statistic $\tau$ satisfies the strict inequality $|\tau| < \frac{1}{3}$. \end{lemma}
\begin{proof} Since $x > 1, 1 - 2x < -1$, and in particular, $1-2x$ is strictly negative. Then from the definition of $\tau,$

$$\tau = \dfrac{\dfrac{1+x+y}{3} - x}{y}= \dfrac{1 + x + y - 3x}{3y} =\dfrac{1 - 2x + y}{3y} = \dfrac{1}{3} + \dfrac{1-2x}{3y}< \dfrac{1}{3}.$$

To show the lower bound for $\tau$, start with the simplified expression $\tau = \dfrac{1 - 2x + y}{3y}$ and notice that $\tau$ decreases as $x$ increases. Since $x < y$, as $x \rightarrow y^-, \tau \rightarrow \dfrac{1-2y+y}{3y} = \dfrac{1-y}{3y} = \dfrac{1}{3y}-\dfrac{1}{3}.$ As $y \rightarrow \infty$, $\tau \rightarrow -\dfrac{1}{3}$. Thus $\tau > -\frac{1}{3}$ for any $1<x<y$.

\end{proof}

\begin{remark}\label{xtau_meaning}
The value of $\tau$ effectively measures the relative position of the median $x$ between the minimum and the maximum values. When $\tau$ is close to its upper bound of $\frac{1}{3}$, $x$ is as small as possible (close to 1.) This represents a right-skewed population distribution. Similarly, when $\tau$ is close to its lower bound of $-\frac{1}{3}$, $x$ is as large as possible (close to $y$), and represents a left-skewed population vector. When $\tau = 0$, the mean and median are equal, and from the formula for $y$ in terms of $x$ and $\tau$, $y = 2x-1$. Equivalently, $x = \frac{1 + y}{2}$, the average of 1 and $y$. 

The initial space of population vectors is three-dimensional: $(p_1, p_2,p_3).$ Scaling reduces this space to two dimensions: $(1,x,y)$. Introducing $\tau$, a measure of skewness, replaces $y$, a measure of the relative size of the largest state to the smallest state. 
 
\end{remark}

\subsection{Asymptotic Stability and the Apportionment Limit}
Since the reduced population vector $(1,x,y(x,\tau))$ depends continuously on $x$ (and on $\tau$), the next theorem will show the limiting behavior as $x \to \infty$ is determined entirely by taking the limit of the vector of standard quotas, except possibly at the finitely many $\tau$-values where priority values become equal. 

\begin{theorem}\label{tau_limit_thm1}(Asymptotic Apportionment) Let $(1,x,y(x,\tau))$ be a reduced population vector associated with a fixed $\tau$, and let $A$ be a divisor method that guarantees nonzero allocations on $M$ seats. Then, for all but finitely many $\tau \in (-\frac13,\frac13)$: $$\displaystyle{\lim_{x \to \infty} A(1,x,y(x,\tau)) = \tilde{A} \left( \frac{M(1-3\tau)}{3-3\tau},\frac{2M}{3-3\tau} \right)},$$ where $\tilde{A}$ is defined as apportioning the remaining $M-1$ seats among states 2 and 3 according to the divisor method $A$ (see definition \ref{stealing_def}).
\end{theorem}

\begin{proof}
For a fixed $\tau$ and reduced population vector $(1,x,y(x,\tau))=\left( 1,x,\frac{2x-1}{1-3\tau} \right)$, the total population is $P = 1 + x + \frac{2x-1}{1-3\tau} = \frac{(3-3\tau)x - 3\tau}{1-3\tau}.$ Calculate the standard quotas $q_i = \frac{Mp_i}{P}$:
$$(q_1, q_2, q_3) = \left( \frac{M(1-3\tau)}{(3-3\tau)x-3\tau} , \frac{M(1-3\tau)x}{(3-3\tau)x-3\tau} , \frac{M(2x-1)}{(3-3\tau)x-3\tau} \right)$$

\noindent Then, as $x \to \infty$
$$ (q_1, q_2, q_3) \to\left( 0 , \frac{M(1-3\tau)}{3-3\tau} , \frac{2M}{3-3\tau} \right). $$

Divisor methods are piecewise constant functions on open regions of the population space. The boundaries of these regions are hypersurfaces defined by equalities of the priority values, $\frac{p_i}{\delta(k)} = \frac{p_j}{\delta(l)}$, with $0 \leq k,l \leq M-1$. The boundaries cause a discontinuity in $A$ and a change in apportionment. The reduced population vector $(1,x,y)$ depends continuously on $\tau$ (and on $x$). 
There are finitely many such equations, each rational functions in $\tau$. After clearing denominators, each condition yields a linear equation in $\tau$ and has at most one solution. Therefore, the exceptional set of $\tau$ is finite. Define this set as $\mathcal{E}$.

For $\tau \not \in \mathcal{E}$, $A$ is constant on a neighborhood near $\left( 0 , \frac{M(1-3\tau)}{3-3\tau} , \frac{2M}{3-3\tau} \right)$, and therefore
$$
\begin{array}{cl}
\displaystyle\lim_{x \to \infty}A(1, x,y(x,\tau)) & = \displaystyle\lim_{x \to \infty}A(q_1, q_2,q_3) \\ [9pt]
 & =A(\displaystyle\lim_{x \to \infty}(q_1, q_2,q_3)) \\ [9pt]
 & =A \left( 0 , \frac{M(1-3\tau)}{3-3\tau}, \frac{2M}{3-3\tau} \right).
\end{array}
$$
As $A$ guarantees strictly positive apportionments, state one receives at least one seat. Because $q_1(x)\to 0$, state 1's priorities become negligible so state one only receives one seat and the remaining $M-1$ seats are apportioned exactly as if $A$ were applied only to states $2$ and $3$. This is precisely the definition of $\tilde{A} \left( \frac{M(1-3\tau)}{3-3\tau},\frac{2M}{3-3\tau}\right),$ completing the proof.
\end{proof}

In what follows, we introduce new notation for the limiting standard quotas: 
$$
\tilde{q}_2 :=\frac{M(1-3\tau)}{3-3\tau}, \qquad 
\tilde{q}_3 :=\frac{2M}{3-3\tau}.
$$

For $\tau \not \in \mathcal{E}$, both $\tilde{q}_2$ and $\tilde{q}_3$ are nonintegral limit points of the standard quotas $q_2(x)$ and $q_3(x)$. As such, their integer parts must stabilize. Consequently, the asymptotic quota behavior of the full three-state apportionment is determined entirely by which of states 2 or 
3 receives the final seat from $\tilde{A}.$ The following theorem formalizes this stabilization and shows that, beyond a finite threshold, quota compliance or violation is completely determined by $\tau$.

\begin{theorem}\label{X_tau_def/exist} (Asymptotic Quota Stabilization for Nonexceptional $\tau$) Let $\mathcal{E}$ be the exceptional set of finite $\tau \in (-\frac 13,\frac 13)$ such that Theorem \ref{tau_limit_thm1} fails. Fix $\tau \notin \mathcal{E}$. Then there exists some $x_\tau$ such that for all $x \geq x_\tau$, $A(1,x,y(x,\tau))$ is always a lower quota violation caused by nonzero allocation or always not a violation, with this behavior determined by $\tau$. 

Specifically, let state $i \in \{2,3\}$ be the state that receives the last seat when apportioning $M$ seats among only states $2$ and $3$ from $A(\tilde{q}_2, \tilde{q}_3)$. 
Then: 
\begin{enumerate}
 \item If $A(\tilde{q}_2, \tilde{q}_3)$ assigns state $i$ its lower quota 
 $\lfloor \tilde{q}_i \rfloor$, then for all $x \geq x_\tau$, the apportionment 
 $A(1, x, y(x,\tau))$ is a lower quota violation caused by nonzero allocations, 
 with state $i$ receiving $\lfloor \tilde{q}_i \rfloor - 1$ seats.
 \item If $A(\tilde{q}_2, \tilde{q}_3)$ assigns state $i$ its upper quota 
 $\lceil \tilde{q}_i \rceil$, then for all $x \geq x_\tau$, the apportionment 
 $A(1, x, y(x,\tau))$ is not a quota violation.
\end{enumerate}
\end{theorem}

\begin{proof}
Theorem \ref{tau_limit_thm1} implies 
$$ \lim_{x \to \infty} A(1,x,y(x,\tau)) = \tilde{A}\left( \frac{M(1-3\tau)}{3-3\tau} , \frac{2M}{3-3\tau} \right) := \tilde{A}(\tilde{q}_2, \tilde{q}_3) $$
By the definition of a limit, there must exist some point in which $A(1,x,y(x,\tau))$ is epsilon close to $\tilde{A}(\tilde{q}_2,\tilde{q}_3)$. Since $A$ and $\tilde{A}$ only output integer vectors in $\mathbb{Z}^3$, convergence implies there exists some $x_{\tau}^A$ such that 
$$A(1,x,y(x,\tau)) = \tilde{A}(\tilde{q}_2,\tilde{q}_3) \qquad \text{for all } x \geq x^A_{\tau}.$$
Since $\tau \notin \mathcal{E}$, the limiting quotas
$$\tilde{q}_2=\frac{M(1-3\tau)}{3-3\tau}, \qquad 
\tilde{q}_3=\frac{2M}{3-3\tau}$$
are not integers so each lies at positive distance from the nearest integer. Because the corresponding standard quotas $(q_2(x),q_3(x))$ converge to $(\tilde{q}_2,\tilde{q}_3)$ as $x\to\infty$, there exists $\delta>0$ and $x_{\tau}^F$ such that for all $x \ge x_{\tau}^F$ the varying quotas remain within $\delta$ of their limits and therefore lie in the same unit intervals as $\tilde{q}_2$ and $\tilde{q}_3$. It follows that for all $x \ge x_{\tau}^F$,
$$\lfloor q_2(x) \rfloor = \lfloor \tilde{q}_2 \rfloor 
\qquad\text{and}\qquad
\lfloor q_3(x) \rfloor = \lfloor \tilde{q}_3 \rfloor,$$
so the integer parts of the quotas stabilize.

Define $x_\tau = \max\{x_{\tau}^A,x_{\tau}^F\}$. Then for all $x \geq x_\tau$, 
$$A(1,x,y(x,\tau)) = \tilde{A}(\tilde{q}_2,\tilde{q}_3) \text{ and } 
\lfloor q_i(x) \rfloor = \lfloor \tilde{q}_i \rfloor, i = 2,3.$$
Therefore whether $A(1,x,y(x,\tau))$ violates lower quota is equivalent to whether the limiting apportionment does. 

Finally, $A(\tilde{q}_2,\tilde{q}_3) = (a_2,a_3)$ apportions $M-1$ seats, and $\tilde{A}$ assigns the remaining seat to state 1. Let state $i$ receive the final seat in the apportionment of $M$ seats in $A(\tilde q_2, \tilde q_3)$. Then, since for two states there are no quota violations by Theorem \ref{no_vio_n=2} there are only two possible cases:
\begin{enumerate}
\item If $a_i = \lfloor \tilde{q}_i \rfloor$, the final apportionment for state $i$ is $\lfloor \tilde{q}_i \rfloor -1$, which is a lower quota violation caused by nonzero allocations.
\item If $a_i = \lceil \tilde{q_i} \rceil$, the final apportionment for state $i$ is $\lfloor \tilde{q}_i \rfloor$, which is not a quota violation.
\end{enumerate}

Because the limiting quotas $(\tilde{q}_2, \tilde{q}_3)$ are functions of $\tau$, and the two-state allocation $A(\tilde{q}_2, \tilde{q}_3)$ depends only on these values, the presence or absence of a lower quota violation is uniquely determined by $\tau$ as claimed.
\end{proof}

This theorem then motivates the following definitions: 
\begin{definition}\label{Violatory_tau_def}
A value $\tau \notin \mathcal{E}$ is said to be \textit{ultimately-violatory} if for $x \geq x_\tau$, $A(1,x,y(\tau, x))$ produces a quota violation. Otherwise $\tau$ is said to be \textit{ultimately non-violatory} or \textit{ultimately quota compliant}.
\end{definition} 
Additionally, note the structure of $x_\tau$:
\begin{remark}\label{xtau_structure}
The $\tau$-interval $(-\frac{1}{3}, \frac{1}{3})$ can be decomposed into  finitely many intervals such that on each interval, $x_\tau$ is locally represented in the form $x_\tau = \frac{a}{b+c\tau} + \frac{d\tau}{b+c\tau} + e$ for $a,b,c,d,e \in \mathbb{R}$ and denominators not identically $0$. This is proven later in Theorem \ref{for_remark}.

\end{remark}

\subsection{Extension to $n$ States}
For three states, the statistic $\tau$ parametrizes one-dimensional families of reduced population vectors $(1,x,y)$ along which the relative proportions of the populations vary linearly. These families appear as straight rays in the $(x,y)$-plane and are useful because apportionment behavior stabilizes along such rays as $x \to \infty$.

This idea may be generalized to $n$ states, resulting in a similar theory which is presented briefly here. Note that proofs in this section are omitted due to similarity with the three state case.

\begin{definition} 
Given a reduced population vector $(1,x, y_1 \dots, y_{n-2})$, define $\bm{\tau}$ vector as:
\[ \bm{\tau} = (\tau_1, \dots, \tau_{n-2}) = \left( \frac{\frac{1+x+y_1}{3}-x }{y_1}, \dots, \frac{\frac{1+x+y_{n-2}}{3}-x }{y_{n-2}} \right).  \]
\end{definition}
\begin{lemma}
The vector $\bm{\tau} = (\tau_1, \dots, \tau_{n-2})$ satisfies $-\frac{1}{3}<\tau_1 < \dots < \tau_{n-2} < \frac{1}{3}$ under the ordering $(1 < x < y_1 < \ldots < y_{n-2})$.
\end{lemma}
\begin{lemma}
    Given $\bm{\tau} = (\tau_1, \dots, \tau_{n-2})$, $\bm{\tau}$ parameterize a family of population vectors 
    \[(1,x,y_1(x, \tau_1), \dots, y_{n-2}(x,\tau_{n-2})), \qquad \text{where } \qquad y_i(x,\tau_i) = \frac{2x-1}{1-3\tau_i}. \]
\end{lemma}
\begin{theorem}\label{lim_1_nstates}
(Asymptotic Apportionment) Let $(1,x,y_1(x,\tau_1),\ldots,y_{n-2}(x,\tau_{n-2}))$ be a reduced population vector with $\bm\tau=(\tau_1,\ldots,\tau_{n-2})$. Let $A$ be a divisor method that guarantees nonzero allocations on $M$ seats. Then for almost every $\bm\tau$ satisfying $-\tfrac13 < \tau_1 < \cdots < \tau_{n-2} < \tfrac13,$

$$\lim_{x\to\infty}A(1,x,y_1(x,\tau_1),\dots,y_{n-2}(x,\tau_{n-2}))=\tilde A_1(\tilde q_2,\dots,\tilde q_n).$$

The limiting quotas are 
$$\tilde q_2=\lim_{x\to\infty}\frac{Mx}{P(x)},\qquad\tilde q_{i+2}
=\lim_{x\to\infty}\frac{M y_i(x,\tau_i)}{P(x)},$$
where $P(x) = 1+x+y_1(x,\tau_1)+\cdots+y_{n-2}(x,\tau_{n-2}).$
\end{theorem}

\begin{theorem}\label{quota_stabilization_n}
(Asymptotic Quota Stabilization) Let $\bm\tau$ be a nonexceptional parameter vector for which Theorem~\ref{lim_1_nstates} holds. Then there exists $x_{\bm{\tau}}$ such that for all $x\ge x_{\bm\tau}$

$$
A(1,x,y_1(x,\tau_1),\dots,y_{n-2}(x,\tau_{n-2}))
= \tilde A_1(\tilde q_2,\dots,\tilde q_n).
$$

Consequently the quota behavior stabilizes: for sufficiently large $x$, the apportionment is always either quota compliant, a quota violation caused by nonzero allocation, a quota violation worsened by nonzero allocation, or an upper quota violation. Moreover, this behavior depends only on $\bm{\tau}$. If $A(\tilde q_2,\dots,\tilde q_n)=(\tilde a_2,\dots,\tilde a_n)$ is the apportionment of $M$ seats among states $2,\dots,n$ and state $i$ receive the last seat then:
\begin{enumerate}
\item If $A(\tilde q_2,\dots,\tilde q_n)$ is not a quota violation and $\tilde a_i=\lfloor\tilde q_i\rfloor$, then the asymptotic apportionment is a quota violation caused by nonzero allocation.
\item If $A(\tilde q_2,\dots,\tilde q_n)$ is not a quota violation and $\tilde a_i=\lceil\tilde q_i\rceil$, then the asymptotic apportionment is not a quota violation.
\item If $A(\tilde q_2,\dots,\tilde q_n)$ is an upper quota violation where state $i$ is the only state assigned more than its upper quota and $\tilde{a}_i = \lceil\tilde q_i\rceil +1$, then the asymptotic apportionment is not a quota violation.
\item If $A(\tilde q_2,\dots,\tilde q_n)$ is an upper quota violation and case (3) does not apply, then the asymptotic apportionment is an upper quota violation.
\item If $A(\tilde q_2,\dots,\tilde q_n)$ is a lower quota violation, then the asymptotic apportionment is a quota violation worsened by nonzero allocation.
\end{enumerate}
\end{theorem}

\section{Probability of Quota Violations Caused by Nonzero Allocation in Three States }
In this section, $\tau$ is used to derive a formula for probability of a quota violation caused by nonzero allocations for $n=3$ states. We begin by studying the case in which reduced population vectors are drawn uniformly from $\{(1,x,y) \mid 1<x<y \}$, before expanding the analysis to more general distributions.

Throughout, $A$ denotes either the Huntington-Hill, Adams, Dean, modified Webster or modified Jefferson's method. The argument presented here can be extended to additional divisor methods with minor modifications. 

\subsection{The Asymptotic Uniform Distribution}
The following result characterizes $\tau$ when a reduced population vector $(1,x,y)$ is drawn from the region $\{(1,x,y) \mid 1<x<y \}$ under an asymptotic uniform distribution. Heuristically the drawing is from $\{(1,x,y) \mid 1<x<y<h \}$ where $h$ is large. Formally, the definition of the probability for a set $E$ is
$$ P(E) =
\lim_{h\to\infty} P_h(E)
$$
where $P_h(E)$ is the probability of $E$ under a uniform distribution on $\{(1,x,y) \mid 1<x<y<h\}$ and provided the limit exists. 
\begin{lemma}\label{prob_given_tau_uniform} Under the asymptotic uniform distribution defined above on reduced population vectors $\{(1,x,y) \mid 1<x<y \}$, the parameter $\tau$ is a uniform random variable on $\left( -\frac{1}{3}, \frac{1}{3}\right)$. \end{lemma}

\begin{proof} It suffices to show that for every $k \in \mathbb{R}$, the tail probability $P(\tau > k)$ coincides with that of a Uniform$\left(-\frac13,\frac13 \right)$ random variable. The only unknown probability is when $\tau \in (-\frac13 , \frac13)$ by Lemma \ref{t_thirds}: 
$$P(\tau > k) = \left\{ \begin{array}{lcr}
1 & k \leq -\frac{1}{3} \\ ? & k \in \left( -\frac{1}{3}, \frac{1}{3} \right) \\ 0 & k \geq \frac{1}{3} \end{array} \right.$$
To find the unknown expression, let $k \in \left( -\frac{1}{3}, \frac{1}{3} \right)$. Then
$$ P\left(\tau>k\right) = P\left(\frac{\frac{1+x+y}{3}-x}{y}>k\right) = P\left(x<\left(\frac{1}{2}-\frac{3}{2}k\right)y + \frac12 \right). $$
This can then be interpreted geometrically as the fraction of the region $1<x<y$ covered by $x<\left(\frac{1}{2}-\frac{3}{2}k\right)y + \frac12$. For small $y$, the region $x<\left(\frac{1}{2}-\frac{3}{2}k\right)y + \frac12$ may lie below $x = 1$, so the lower bound in $y$ is found by solving $\left(\frac{1}{2}-\frac{3}{2}k\right)y + \frac12 = 1$, which gives $y = \frac{1}{1-3k}.$ For $k \in \left(-\frac13,\frac13\right)$, the line 
$x=\left(\frac12-\frac32k\right)y+\frac12$ has slope strictly less than $1$, so for sufficiently large $y$ it lies below the line $x=y$. For sufficiently large $y$, the line determines the upper boundary in $x$ and the contribution from the region where $x=y$ bounds the domain is asymptotically negligible. Accordingly, the $y$–integration extends to $h$.

Thus, the desired probability equals 
\begin{align*}
P(\tau > k) 
 = P\left(x<\left(\frac{1}{2}-\frac{3}{2}k\right)y+ \frac12\right) 
 = \lim_{h \to \infty} 
 \frac{ 
 \int_{\frac{1}{1-3k}}^h 
 \left(\frac{1}{2}-\frac{3}{2}k\right)y - \frac{1}{2} \,dy }{\int_1^h (y -1) \,dy}
 =\frac{1}{2} - \frac{3}{2}k
\end{align*}

The last equality follows since both the numerator and the denominator are quadratic polynomials in $h$, so the limit is determined by the leading coefficients. For $k \in \left(-\frac13,\frac13\right)$, this coincides with the tail probability of a Uniform$\left(-\frac13,\frac13\right)$ random variable.
\end{proof}

The uniform distribution of $\tau$ obtained above is not merely an result of the specific area normalization on the wedge $\{1<x<y\}$. In fact, the above calculation reveals that the distributional behavior of $\tau$ depends primarily on the asymptotic distribution of the ratio $x/y$. The following proposition makes this precise.

\begin{prop}[Robustness of the Uniform Distribution of $\tau$]
\label{tau_ratio_uniform}
Suppose reduced population vectors are sampled for which $u=\frac{x}{y}$ is asymptotically uniformly distributed on $(0,1)$ and $y \to \infty$ in density.
Then the induced distribution $\tau = \frac{1-2x+y}{3y}$ converges to the uniform distribution on $\left(-\frac13,\frac13\right)$.
\end{prop}

\begin{proof}
Writing $x=uy$, we have $\tau = \frac{1-2uy+y}{3y}= \frac{1}{3y} + \frac{1-2u}{3}.$
As $y \to \infty$, the term $\frac{1}{3y}$ vanishes in probability. 
Thus the limiting distribution of $\tau$ coincides with that of 
$\frac{1-2u}{3}.$
Since $u$ is asymptotically uniform on $(0,1)$, it follows that $\frac{1-2u}{3}\sim \mathrm{Unif}\!\left(-\frac13,\frac13\right)$.
\end{proof}

Having established that $\tau$ is uniformly distributed on $\left(-\tfrac13,\tfrac13\right)$, we now relate the behavior of quota violations to the distribution of $\tau$. The following lemma shows that the problem reduces to identifying which values of $\tau$ are ultimately violatory.

\begin{theorem}[Quota Violations Under the Asymptotic Uniform Distribution]
\label{zero_one_quota}
Under the asymptotic uniform distribution on reduced population vectors $(1,x,y)$ with $1<x<y$ the following are equivalent:
\begin{itemize} 
\item $P(A(1,x,y) \text{ is a quota violation}). $
\item $P(A(1,x,y) \text{ is a quota violation caused by nonzero allocations)}.$ 
\item $P(\tau(x,y) \text{ is ultimately violatory}).$
\end{itemize}
\end{theorem}

\begin{proof}
It follows from  Theorem \ref{X_tau_def/exist} and the definition of ultimately violatory that the following are equivalent:
\begin{itemize} 
\item $P(A(1,x,y) \text{ is a quota violation for $x>x_\tau$}) $
\item $P(A(1,x,y) \text{ is a quota violation caused by nonzero allocations for $x>x_\tau$)}$ 
\item $P(\tau(x,y) \text{ is ultimately violatory for $x>x_\tau$}).$
\end{itemize}
As such, it suffices to show $P(x<x_\tau) = 0$, as then the desired statement may be deduced by adding $0$. \\
\\
To start, let $P_h(E)$ be the probability of $E$ under the uniform distribution on $\{(1,x,y) \mid 1<x<y<h \}$ and let $r \in (1, h)$. Then:
\[ 0 \leq P_h(x<x_\tau) \leq P_h(\tau \mid x_\tau > r) + P_h( x< r ). \]
It follows from Remark \ref{xtau_structure} that for each $\varepsilon > 0$, there exists $r_\varepsilon$ such that the total length of the set of $\tau$ such that $x_\tau > r_\varepsilon$ is less than $\left(\frac{2}{3} \right) \cdot\left( \frac{\varepsilon}{2} \right)$. Then, by Lemma \ref{prob_given_tau_uniform}, $P_h(x_\tau > r_\varepsilon)$ converges to $\frac{3}{2}$ times the total length of such $\tau$ and therefore for all $h$ sufficiently large:
\[ P_h(\tau \mid x_\tau > r_\varepsilon ) < \frac{\varepsilon}{2}. \]
Additionally compute using the area of triangles and squares:
\[ P_h(x<r_\varepsilon) = \frac{2(h-1)(r_\varepsilon -1) - (r_\varepsilon -1)^2}{(h-1)^2} \]
which goes to $0$ as $h \to \infty$. Thus, for all $h$ sufficiently large:
\[ P_h( x< r_\varepsilon ) < \frac{\varepsilon}{2}. \]
Then for all $h$ sufficiently large:
\[ 0 \leq P_h(x<x_\tau) \leq P_h(\tau \mid x_\tau > r_\varepsilon) + P_h( x< r_\varepsilon ) < \varepsilon. \]
So
\[ P(x<x_\tau) = \lim_{h \to \infty} P_h(x<x_\tau) = 0, \]
as needed.

\end{proof}

\subsection{An Asymptotic Probability Function for Three States}
Fix $M$ seats and $n=3$ states. By Theorem \ref{zero_one_quota}, to derive a probability function for quota violations under uniformly distributed population vectors, it suffices to determine precisely which values of $\tau$ cause families to be violatory for all $x \geq x_\tau$. Once the subset of admissible $\tau$ is determined, the desired probability is then the length of this subset divided by the length of all possible $\tau$, which is $\frac{2}{3}$ by Lemma \ref{t_thirds}. 

As shown in Section 4, the three-state problem reduces to a two-state problem. In particular, there is a bijection between $\tau$ and the standard quota $\tilde{q}_3.$ For each $\tau$, the reduced population vector $(1,x,y(x,\tau))$ has the same eventual apportionment behavior (for $x \geq x_\tau$) to that of $(\tilde{q}_2,\tilde{q}_3)$ by Theorem \ref{X_tau_def/exist}. 

The relationship between $\tau$ and $\tilde{q}_3$ are given by Theorem \ref{tau_limit_thm1}: 
$$ \tilde{q}_3 = \frac{2M}{3-3\tau}, \qquad \tilde{q}_2 + \tilde{q}_3 = M$$
Consequently,
$$
\tau = 1 - \frac{2M}{3\tilde{q}_3}.
$$

Thus specifying $\tilde{q}_3$ determines $\tau$ and $\tilde{q}_2$, and vice versa. By construction, $\tilde{q}_3>\tilde{q}_2$.

By Theorem \ref{X_tau_def/exist}, a value of $\tau$ produces a lower quota violation for the reduced population vector if and only if the corresponding two state apportionment $\tilde{A}(\tilde{q}_2, \tilde{q}_3)=\tilde{A}(M-\tilde{q}_3, \tilde{q}_3)$ produces a lower quota violation. We now determine precisely for which values of $\tilde{q}_3$ a lower violation occurs. 

Write $\tilde{q}_3 = \lfloor \tilde{q}_3 \rfloor + d$ for $d \in [0,1)$. The occurrence of a lower quota violation is determined by the apportionment of the final seat as $d \to 1^-$. Specifically, by comparing the priority values that determine if the last seat is assigned to the state at $\lfloor \tilde{q}_2\rfloor$ or $\lceil \tilde{q}_2\rceil$, then we can determine exactly when a lower quota violation occurs in $A$. When $A$ assigns the final seat to the first state, if the seat is removed from either the floor of $\tilde{q}_2$ or $\tilde{q}_3$ then a lower quota violation will occur. If the last seat is removed for the first state from either the ceiling of $\tilde{q}_2$ or $\tilde{q}_2$, then no lower quota violation will occur. 

\begin{lemma}\label{d.initial=integers} Using the notation from above, when $\tilde{q}_3 = \lfloor \tilde{q}_3\rfloor$ (or equivalently when $d=0$), there is lower quota violation in $\tilde{A}(M-\tilde{q}_3, \tilde{q}_3)$. \end{lemma}

\begin{proof}
Suppose $d = 0$. Then $\tilde{q}_3 = \lfloor \tilde{q}_3 \rfloor$ is an integer. Since $\tilde{q}_2 = M - \tilde{q}_3$, it follows that $\tilde{q}_2$ is also an integer. When both quotas are integers, the two state apportionment agrees exactly with the quota vector: $\tilde{A}(\tilde{q}_2,\tilde{q}_3) = (\tilde{q}_2,\tilde{q}_3).$

Therefore, in the corresponding three–state apportionment, state 1 must receive one seat, and one seat must be removed from either state 2 or state 3. Because $\tilde{q}_2$ and $\tilde{q}_3$ are integers, removing one seat from either state forces that state to receive strictly fewer seats than its lower quota. Therefore a lower quota violation occurs in $\tilde{A}(M-\tilde{q}_3,\tilde{q}_3)$.
\end{proof} 

\begin{lemma}\label{d.intervals} Let $\tilde{q}_3 = \lfloor \tilde{q}_3 \rfloor + d$, for $d\in [0,1)$. 
There exist numbers
$$
0 \le d^* \le d^{**} \le 1
$$
such that $[0,1)$ decomposes into three intervals
$$I_1 = [0,d^*), \qquad I_2 = [d^*,d^{**}),\qquad I_3 = [d^{**},1),$$
with the following properties:

\begin{enumerate}
\item For $d \in I_1$, a lower quota violation occurs in $\tilde{A}(M-\tilde{q}_3, \tilde{q}_3)$.
\item For $d \in I_2 \cup I_3$, no lower quota violation occurs in $\tilde{A}(M-\tilde{q}_3, \tilde{q}_3)$.
\end{enumerate}
\end{lemma}

\begin{proof}
By Lemma \ref{d.initial=integers}, when $d=0$, a lower quota violation occurs. As $d$ increases, a lower quota violation continues to occur while the third state's priority values increase and the second state's priority values decrease. For these values of $d$ close to zero, $a_2 = \lceil\tilde{q}_2 \rceil$ and $a_3 = \lfloor \tilde{q}_2 \rfloor $, with $a_3$ assigned last by $A$. Eventually, for some $d$, the priority value that puts state three at its lower quota becomes less than the priority value that puts state two at its upper quota. Define $d^*$ to be the first value of $d$ where this happens and define the interval $I_1 = [0,d^*)$. For $d \in I_1$, $a_2 = \lceil\tilde{q}_2 \rceil$ and $a_3 = \lfloor \tilde{q}_3 \rfloor $, with $a_3$ assigned last by $A$. Then $\tilde{A}$ results in the final apportionment with $a_3= \lfloor \tilde{q}_3 \rfloor-1$, a lower quota violation.

For $d\geq d^*$, $A$ assigns $a_2 = \lceil\tilde{q}_2 \rceil$ and $a_3 = \lfloor \tilde{q}_3 \rfloor $, with $a_2$ assigned last, until there is some $d$ that causes state three to be apportioned its upper quota and state two to be apportioned its lower quota, with the last seat assigned to state three by $A$. Call the first value of $d$ where this happens $d^{**}$ and the interval $I_2 = [d^*,d^{**})$. For $d \in I_2$, $a_2 = \lceil \tilde{q}_2 \rceil$ and $a_3 = \lfloor \tilde{q}_3 \rfloor$, with $a_2$ assigned last by $A$. Then $\tilde{A}$ results in the final apportionment with $a_2=\lfloor \tilde{q}_2 \rfloor$, which is a not lower quota violation. %

Define $I_3 = [d^{**},1).$ For $d \in I_3$, $a_2 = \lfloor \tilde{q}_2 \rfloor$ and $a_3 = \lceil \tilde{q}_3 \rceil$ with $a_3$ assigned last by $A$. Then $\tilde{A}$ results in the final apportionment with $a_3= \lfloor \tilde{q}_3 \rfloor$, which is a not lower quota violation. %

Thus lower quota violations occur precisely for $d \in I_1=[0,d^*)$.
\end{proof}

We summarize the results of this lemma with the below table: 

\begin{center}
\begin{tabular}{ |c|c|c|c|c| } 
\hline
Interval & $A$ assigns $a_2$ & $A$ assigns $a_3$ & Assigned Last & Lower Quota Violation?\\
\hline
$I_1 = [0, d^*)$ & $\lceil \tilde{q}_2 \rceil$ & $\lfloor \tilde{q}_3 \rfloor$ &$a_3$ & Yes \\
$I_2 = [d^*,d^{**})$ & $\lceil \tilde{q}_2 \rceil$ & $\lfloor \tilde{q}_3 \rfloor$ &$a_2$ & No \\
$I_3 = [d^{**},1)$ & $\lfloor \tilde{q}_2 \rfloor$ & $\lceil \tilde{q}_3 \rceil$ &$a_3$ & No\\
\hline
\end{tabular}
\end{center}

\begin{remark} 
It may occur that $d^*=0$, in which case $I_1=\emptyset$. It may also occur that $d^{**}=1$, in which case $I_3=\emptyset$.
\end{remark}

\begin{remark}\label{d.noCase4}
By Theorem \ref{36}, there is no value of $d \in (0,1)$ for which the final seat is assigned to state 2 while $a_2=\lfloor \tilde q_2 \rfloor$.
\end{remark}

In summary, for each fixed $\tilde q_3$, a lower quota violation occurs if and only if $d \in [0,d^*).$
Equivalently, violations occur precisely when $\tilde q_3$ lies in the interval $(\lfloor \tilde q_3 \rfloor,\; \lfloor \tilde q_3 \rfloor + d^*),$ and these hold within each fixed integer interval. 

\subsubsection{A formula for $d^*$}
To find the interval length ending at $d^*$, a function is required that determines $d^*$ given $\lfloor \tilde q_3 \rfloor$.
\begin{deflem}\label{d.star} There exists a function $D(x)$ that calculates $d^*$. \end{deflem}
\begin{proof}
By the interval definitions in Lemma \ref{d.intervals}, the priority value associated with the larger state's final seat is less than the priority value associated with the smaller state's final seat for $d < d^*$. Therefore, $d^*$ occurs precisely when the priority values are equal:
$$ \frac{\lfloor \tilde{q}_3 \rfloor +d^*}{\delta(\lfloor \tilde{q}_3 \rfloor -1)} = \frac{M-\lfloor \tilde{q}_3 \rfloor-d^*}{\delta(M-1-\lfloor \tilde{q}_3 \rfloor)}$$

Solving for $d^*$ yields:
$$ d^* = \frac{(M-\lfloor \tilde{q}_3 \rfloor)\delta(\lfloor \tilde{q}_3 \rfloor-1) - \lfloor \tilde{q}_3 \rfloor \delta(M-\lfloor \tilde{q}_3 \rfloor-1)}{\delta(\lfloor \tilde{q}_2 \rfloor-1) +\delta(M-\lfloor \tilde{q}_2 \rfloor-1)} $$

Since divisor functions satisfy $\delta(k) > 0$ for all admissible $k$, the denominator is strictly positive, and $d^* \in [0,1)$.

This then lets us define the function $D(k): \mathbb{N} \to \mathbb{R}$, for $M$ and $\delta(x)$ given, by 
$$D(k) := \frac{(M-k)\delta(k-1) - k \delta(M-k-1)}{\delta(k-1) +\delta(M-k-1)}$$
\end{proof}

\begin{remark}
Note, while $D(k)$ is defined for any $k \in \mathbb{R}$ where $\delta(k)$ is defined and the denominator is not zero, we will only need it for positive integers $k$. In addition, $D(k)$ can output negative values. $d^*$ will only be defined by $D(k)$ when $D(k)\in [0,1).$ If $D(k) <0, I_1 = \emptyset.$
\end{remark}

\begin{lemma}\label{I1.zero}
Suppose $k = \left\lfloor \frac{M}{2} \right\rfloor$.
If $M$ is even, then $d^*=0$ and $I_1=\{0\}$.
If $M$ is odd, then $d^*<0$ and $I_1=\emptyset$.
In either case, $\mathrm{len}(I_1)=0$.
\end{lemma}

\begin{proof} For $M$ even, $\lfloor \frac{M}{2} \rfloor = \frac{M}{2}$. Then, use the function $D$ to calculate $d^*$:
$$ d^* = D \left( \frac{M}{2} \right) =\frac{\left(M-\frac{M}{2}\right) \delta \left(\frac{M}{2}-1\right) - \frac{M}{2} \delta \left(M-\frac{M}{2}-1 \right)}{\delta \left(\frac{M}{2}-1\right)+\delta\left(M-\frac{M}{2}-1 \right)}.$$

Note the numerator is 0 and the denominator is $2\delta(\frac M2 -1)$ is strictly positive, so $d^* = 0$
Thus, $I_1 =[0,d^*) = \emptyset$ and $\text{len} (I_1) = 0$. 

For $M$ odd, $k =\left\lfloor \frac{M}{2} \right\rfloor=\frac{M-1}{2}$. Then $d^* = D(k).$
Since $$D(k) = \frac{(M-k)\delta(k-1)-k\,\delta(M-k-1)}{\delta(k-1)+\delta(M-k-1)},$$
and since $M-k=\frac{M+1}{2}$,and $M-k-1=k$ this simplifies to:
$$D(k) = \frac{\frac{M+1}{2}\, \delta(k-1)-\frac{M-1}{2}\, \delta(k)}{\delta(k-1)+\delta(k)}.$$
The denominator is strictly positive so the sign of $D(k)$ is determined by the numerator
$N := \frac{M+1}{2}\, \delta(k-1) - \frac{M-1}{2}\, \delta(k).$

A direct substitution of the divisor functions yields:
\begin{enumerate}
\item For Adams ($\delta(s)=s$), $N=-1<0$.
\item For Jefferson ($\delta(s)=s+1$), $N=0$.
\item For Webster ($\delta(s)=s+\tfrac12$), $N=-\tfrac12<0$.
\item For Huntington-Hill and Dean, a straightforward computation likewise gives $N<0$.
\end{enumerate}

Thus $D(k)\le 0$ for all five workable methods, with equality occurring only for Jefferson. As such, $d^*\le 0$. Since admissible $d\in[0,1)$, it follows that $
I_1=[0,d^*)= \emptyset$ for $d^*\leq 0$ and therefore $\mathrm{len}(I_1)=0$.
\end{proof}

\begin{remark}\label{d.M/2}
If $\left\lfloor \tilde{q}_2 \right\rfloor = M$, then the associated lower-violation interval $I_1$ has length zero since $D(M) = 0$. If $k < \lceil \frac{M}{2} \rceil$, then $D(k) <0$ and $I_1 = \emptyset.$ Thus, when searching for quota violations, it suffices to consider $\left\lfloor \tilde q_2 \right\rfloor \in [\left \lceil \frac{M}{2} \right \rceil,\, M-1]$.
\end{remark}

\noindent We are now ready to prove the main result of this section: 
\begin{theorem}\label{prob_sum} Let $A$ be a divisor method with divisor function $\delta$ that guarantees nonzero allocations on M seats and $n = 3$ states. Then, under the asymptotic uniform distribution on reduced population vectors $(1, x, y)$, with $1<x<y$,  the probability of a quota violation is the same as the probability of a quota violation caused by states guaranteed nonzero allocations and is calculated by:
$$P(A(1,x,y) \text { has a q.v.}) = \sum_{k = \lceil M/2 \rceil}^{M-1} M\left( \frac{1}{k} - \frac{1}{k + D(k)} \right)$$
where $k$ indexes the integer values of $\lfloor \tilde{q}_3 \rfloor$, where 
$$\tilde{q}_3 = \frac{2M}{3-3\tau}, \qquad \tau = \frac{\mathrm{mean}(1,x,y)-\mathrm{median}(1,x,y)}{\max(1,x,y)}=\dfrac{1 - 2x + y}{3y}.$$ and 
$$D(k) =\frac{(M-k)\delta(k-1)-k\,\delta(M-k-1)}{\delta(k-1)+\delta(M-k-1)}.$$
\end{theorem}

\begin{proof}
Combining the previous lemmas yields the stated expression. By Theorem \ref{zero_one_quota} it suffices to find the probability that a given $\tau$ is ultimately violatory. Since $\tilde{q}_3 = \frac{2M}{3-3\tau}$ is a bijection, fixing $\tau$ then determines $\tilde{q}_3$, and vice versa.

By Lemma \ref{d.intervals}, for a fixed $\tilde{q}_3$, quota violations can only be lower quota violations and they only occur in the interval $(\lfloor \tilde{q}_3 \rfloor, \lfloor \tilde{q}_3 \rfloor + d^*)$. 

Translate this interval into its corresponding values of $\tau$ through the strictly increasing relation $\tau = 1 - \frac{2M}{3\tilde{q}_3}$ to map this interval to $$(1 - \frac{2M}{3\lfloor \tilde{q}_3 \rfloor}, 1-\frac{2M}{3(\lfloor \tilde{q}_3 \rfloor+D(\lfloor \tilde{q}_3 \rfloor))}),$$ where $D(x)$ calculates $d^*$ as defined in Lemma \ref{d.star}. 

The probability of an ultimately-violatory $\tau$ is then the sum of the length of all such intervals for all choices of $\tilde{q}_3 \in \left\{ \left\lceil \frac{M}{2} \right\rceil, \dots, M-1 \right\}$ by Lemma \ref{d.intervals} and Remark \ref{d.M/2} over the length of all possible $\tau$ values, which is $\frac{2}{3}$ by Lemma \ref{t_thirds}. 

The probability is then:
$$\frac{3}{2} \left( \sum_{\lfloor \tilde{q}_3 \rfloor= \lceil \frac{M}{2} \rceil}^{M-1} \left(1-\frac{2M}{3(\lfloor \tilde{q}_3 \rfloor+D(\lfloor \tilde{q}_3 \rfloor))}\right)- \left(1 - \frac{2M}{3\lfloor \tilde{q}_3 \rfloor}\right)\right)$$ 
which simplifies to $$\sum_{\lfloor \tilde{q}_3 \rfloor= \lceil \frac{M}{2} \rceil}^{M-1} M\left( \frac{1}{\lfloor \tilde{q}_3 \rfloor} - \frac{1}{\lfloor \tilde{q}_3 \rfloor + D(\lfloor \tilde{q}_3 \rfloor)} \right).$$
\end{proof}

This formula has a nice closed form for the Modified Jefferson's method: 

\begin{corollary}\label{jeff_closedform}
Let $A$ be the Modified Jefferson method with divisor function
\[
\delta(s) = \begin{cases} 0 & s = 0, \\ s + 1 & s \geq 1. \end{cases}
\]
Then, under the asymptotic uniform distribution on reduced population vectors $(1, x, y)$
with $1 < x < y$, the probability of a quota violation caused by nonzero allocations is
exactly
\[
P_M^{\mathrm{JEFF}} = \frac{1}{M-1}.
\]
\end{corollary}

\begin{proof}
By Theorem \ref{prob_sum}, the probability of a quota violation is
\[
P_M^{\mathrm{JEFF}} = \sum_{k=\lceil M/2 \rceil}^{M-1} M\left(\frac{1}{k} - 
\frac{1}{k + D(k)}\right),
\]
where
\[
D(k) = \frac{(M-k)\,\delta(k-1) - k\,\delta(M-k-1)}{\delta(k-1) + \delta(M-k-1)}.
\]
Since $\delta(k) = k+1$ for $k \geq 1$, $\delta(k-1) = k$ for all $k \geq 1$.
For $k \leq M-2$, $M - k - 1 \geq 1$, so $\delta(M-k-1) = M - k$. Substituting:
\[
D(k) = \frac{(M-k)\cdot k - k\cdot(M-k)}{k + (M-k)} = \frac{0}{M} = 0,
\]
so every summand vanishes for $\lceil M/2 \rceil \leq k \leq M-2$.

For $k = M-1$, $M - k - 1 = 0$ and the modified case $\delta(0) = 0$ applies. Substituting:
\[
D(M-1) = \frac{1 \cdot (M-1) - (M-1) \cdot 0}{(M-1) + 0} = \frac{M-1}{M-1} = 1.
\]
Therefore $k + D(k) = M$, and the summand at $k = M-1$ is:
\[
M\left(\frac{1}{M-1} - \frac{1}{M}\right) = M \cdot \frac{1}{M(M-1)} = \frac{1}{M-1}.
\]
Since all other summands are zero, the total probability is
\[
P_M^{\mathrm{JEFF}} = \frac{1}{M-1},
\]
which confirms the value stated in Section~5.4 and agrees with 
$\lim_{M\to\infty} P_M^{\mathrm{JEFF}} = 0$ as given in Theorem \ref{Mtoinftyproblimit}.
\end{proof}

We would now like to lift this result from the reduced population vector $(1,x,y)$ to the original population vector $(p_1,p_2,p_3)$. Apportionments are proportional, so events like lower quota violations are invariant under scaling, but probabilities are not unless the sampling model is scale invariant. A natural choice then is to sample $(p_1,p_2,p_3)$ from a scale-invariant distribution, e.g. as i.i.d. continuous variables with density depending only on their ratios. Proposition \ref{tau_ratio_uniform} gives the assumption needed: that $\frac{p_2}{p_3}$ be uniform on $(0,1)$. Then $\tau$ will be uniform on $(-\frac13, \frac13)$ and we can lift the probability formulas. Without these assumptions on $(p_1,p_2,p_3)$, the probability results from the prior theorem may no longer apply. We formalize this in the following. 

\begin{corollary}\label{prob_lift}
Let $(p_1,p_2,p_3)$ be a random population vector with $p_1 < p_2 < p_3$, and define $x = \frac{p_2}{p_1}$, $y = \frac{p_3}{p_1}$. Assume that the induced ratios satisfy the hypotheses of Proposition \ref{tau_ratio_uniform}; namely, that
$u = \frac{x}{y} = \frac{p_2}{p_3}$ is asymptotically uniformly distributed on $(0,1)$ and $y \to \infty$ in density.

Let $A$ be a divisor method that guarantees nonzero allocations, and let $\tau$ be as in Definition \ref{tau_def}. Then:
\begin{enumerate}
    \item $\tau \sim \mathrm{Unif}\!\left(-\frac{1}{3}, \frac{1}{3}\right)$ asymptotically.
    \item The probability that $A(p_1,p_2,p_3)$ exhibits a lower quota violation caused by nonzero allocations satisfies
    \[
    P(\text{quota violation}) =P(\text{quota violation caused by nonzero allocation})   = P(\tau \text{ is ultimately violatory}).
    \]
\end{enumerate}
In particular, this probability coincides with that obtained under the asymptotic uniform distribution on reduced population vectors $(1,x,y)$.
\end{corollary}

\begin{proof}
By Lemma \ref{neutrality} (Proportionality), divisor methods are invariant under scaling, so
$A(p_1,p_2,p_3) = A(1,x,y).$ Thus quota violation events depend only on the reduced population vector. By Proposition \ref{tau_ratio_uniform}, the assumptions on $(p_1,p_2,p_3)$ imply that the induced distribution of $\tau$ converges to $\mathrm{Unif}\!\left(-\frac{1}{3}, \frac{1}{3}\right)$.

By Theorem \ref{tau_limit_thm1}, for each $\tau$ outside a finite exceptional set, the occurrence or non-occurrence of a lower quota violation for $A(1,x,y(x,\tau))$ stabilizes for all sufficiently large $x$, and is completely determined by whether $\tau$ is ultimately violatory. Therefore, the probability of a lower quota violation given $\tau$ depends only on $\tau$. As such, $P(\text{lower quota violation})= P(\tau \text{ is ultimately violatory}).$

Since both the present model on $(p_1,p_2,p_3)$ and the asymptotic uniform model on $(1,x,y)$ induce the same limiting distribution on $\tau$, the resulting probabilities coincide.
\end{proof}

\subsection{Numerical Behavior of the Probability Formula}
Table~\ref{tab:probabilities} below presents approximate values of the probability of lower quota violations caused by nonzero allocations for the classical divisor methods and increasing values of $M$. Values are rounded to 3 decimal places. 

\begin{table}[H]
\centering
\begin{tabular}{lcccccc}
\toprule
Method & $M=10$ & $M=20$ & $M=50$ & $M=100$ & $M=500$ & $M = 1000$\\
\midrule
Modified Jefferson & 0.111 & 0.053 & 0.020 & 0.010 & 0.002 & 0.001\\
Adams & 0.380 & 0.385 & 0.386 & 0.386 & 0.386 & 0.386\\
Modified Webster & 0.235 & 0.213 & 0.201 & 0.197 & 0.194 & 0.194\\ 
Huntington-Hill & 0.257 & 0.229 & 0.209 & 0.202 & 0.195 & 0.194 \\
Dean & 0.278 & 0.244 & 0.218 & 0.207 & 0.197 & 0.195\\
\bottomrule
\end{tabular}
\caption{Approximate probabilities of lower quota violations caused by nonzero allocations for several divisor methods, computed using Theorem~\ref{prob_sum}. Values rounded to 3 decimal places.}
\label{tab:probabilities}
\end{table}

Several patterns emerge:
\begin{enumerate}
 \item The probabilities appear to converge as $M \to \infty$ to method-dependent constants. The next section will prove this to be true.  
 
 \item Since the Jefferson method has no lower quota violations, the only quota violations that occur will arise from the modified method. As such, the sum $P_M$ in Theorem \ref{prob_sum} above gives the probability for all lower quota violations for three states. This formula simplifies to $\frac{1}{M-1}$ by Corollary \ref{jeff_closedform}. This shows that as $M$ increases, the limit of the probability for a lower quota violation approaches 0.

 \item Adams method exhibits limiting behavior quickly near 0.386. This value turns out to be exactly $2\ln(2) -1 \approx 0.386294$ as shown in the next section. 
 
 \item The Webster, Huntington-Hill, and Dean's method all exhibit limiting behavior, with limiting values near $0.193$. This limit will turn out to be $\ln(2) - \tfrac12 \approx 0.193147$ as shown in the next section. 
 \end{enumerate}

 \begin{remark}
The probabilities in Table~\ref{tab:probabilities} for the Modified Webster method admit a particularly clean interpretation. By Theorem~4.5(b) of \parencite{Self}, Webster's method is 
the unique divisor method for which no quota violations occur in the three-state case when there are no allocation constraints. Therefore, every quota violation appearing in Table~\ref{tab:probabilities} for Modified Webster is entirely attributable to the nonzero allocation constraint. In this sense, Modified Webster isolates the effect of the constraint most cleanly among the five workable methods: its unmodified violation probability is exactly zero, making the values in Table~\ref{tab:probabilities} a pure measure of the constraint's impact. By contrast, the Adams, Dean, and Huntington-Hill methods have nonzero unconstrained violation probabilities for $n = 3$ states, so their entries reflect a combination of the method's inherent tendency to violate quota and the additional effect of the nonzero allocation constraint.
\end{remark}

\subsection{Asymptotic Limit as $M \to \infty$}
We now derive a closed-form expression for the limiting probability of lower quota violations caused by nonzero allocations in $n=3$ seats for different apportionment methods as $M \to \infty$. 

By Corollary \ref{jeff_closedform}, the Modified Jefferson's method has $P_M = \frac{1}{M+1}$ and therefore $$ \lim_{M\to\infty} P_M^{\mathrm{JEFF}} = 0.$$

Next, we analyze Adams Method as $M\to \infty$. 

\begin{theorem}[Limiting Probability for Adams Method]
\label{thm:adams_limit}
Let $A$ be the Adams apportionment method with divisor function $\delta(x)=x$ and let $M$ denote the total number of seats. Then, under the asymptotic uniform distribution on reduced population vectors $(1,x,y)$ with $1<x<y$, the probability of a quota violation converges as $M\to\infty$ to 
$$
P_\infty^{\text{ADAMS}} = \lim_{M \to \infty} \sum_{k=\lceil M/2 \rceil}^{M-1} M \Biggl( \frac{1}{k} - \frac{1}{k + D(k)} \Biggr) = 2\ln(2) -1\approx 0.386294,
$$
where
$$
D(k) = \frac{(2k-M)}{M-2}.
$$
\end{theorem}

\begin{proof}
When $\delta(k)=k$, $D(k)$ from Theorem \ref{prob_sum} simplifies to $D(k) = \frac{2k-M}{M-2}$ . Define the variable $x = \frac{k}{M} \in \left[\tfrac{1}{2},1\right].$

Then, for large $M$, 
$$D(k) = \frac{2k-M}{M-2} \sim \frac{2k-M}{M} = 2x - 1.$$

The summand in $P_M$ then becomes:
\begin{align*}
M \left( \frac{1}{k} - \frac{1}{k + D(k)} \right) & = M \left( \frac{D(k)}{k(k+D(k)} \right) \\ &\sim M\left( \frac{2x-1}{(Mx)(Mx+2x-1)}\right)\\
 & \sim M\left( \frac{2x-1}{(Mx)(Mx)}\right)\\
 & \sim \frac{2x-1}{x^2}\cdot\frac1M\\
\end{align*}
As $M \to \infty$, the sum over $k$ becomes an integral over $x$, and $\frac 1M$ becomes $dx$:
$$
\sum_{k=\lceil M/2 \rceil}^{M-1} M \left(\frac{1}{k} - \frac{1}{k + D(k)}\right) \sim \int_{1/2}^{1} \left( \frac{2x-1}{x^2} \right) dx.
$$

Evaluate the integral to get 
$$P_\infty^\text{ADAMS} = 2\ln(2)-1\approx 0.386294.$$
\end{proof}

The next result identifies a common asymptotic behavior for a broad class of divisor methods whose divisor functions differ from linear growth by a bounded perturbation. In particular, the theorem shows that the limiting probability of a quota violation depends only on the first-order structure of the divisor function, leading to a common limit shared by several classical methods.

\begin{theorem}[Limiting Probability for Asymptotically Linear Divisor Methods]\label{thm:limit_kplusahalf}
Let $A$ be a divisor method with divisor function satisfying
$$
\delta(k)=k+\frac{1}{2}+O\!\left(\frac{1}{k}\right).
$$
Let
$$
P_M=\sum_{k=\lceil M/2\rceil}^{M-1} M\left(\frac{1}{k}-\frac{1}{k+D(k)}\right),
$$
where
$$D(k)=\frac{(M-k)\delta(k-1)-k\,\delta(M-k-1)}{\delta(k-1)+\delta(M-k-1)}.$$
Then
$$\lim_{M\to\infty} P_M = \ln 2 - \frac{1}{2}.$$

In particular, this holds for the Huntington-Hill, Modified Webster, and Dean method. 
\end{theorem}

\begin{proof}
For the Huntington-Hill method, the divisor function $\delta(k) = \sqrt{k(k+1)}$ has Laurent expansion $k + \frac12 -\frac{1}{8k}+\ldots$. After polynomial division, the divisor function for Dean's method is $\delta(k) = \frac{2k(k+1)}{2k+1}= k + \frac 12 - \frac{1}{2(2k+1}$ . Along with the divisor function for Webster's method, $\delta(k) = k +0.5$, all three apportionment methods satisfy the assumption of the theorem. 

Using the assumption $\delta(k) \sim k + \frac 12$, write
$$\delta(k-1)=k-\frac{1}{2} \qquad \delta(M-k-1)=(M-k)-\frac{1}{2}.$$

Substitute into $D(k)$ to get 
$$
D(k) =\frac{(M-k)(k-\frac{1}{2}) - k((M-k)-\frac{1}{2})}{(k-\frac{1}{2})+((M-k)-\frac{1}{2})} =\frac{2k-M}{2(M-1)}$$

Since $D(k)<1$ and $k>1, \frac{D(k)}{k}<1$, by the geometric series, 
$$
 \frac{1}{k+D(k)} = \frac1k \cdot \frac{1}{1 + \frac{D(k)}{k}} = \frac1k \left(1 -\frac{D(k)}{k} + \left(\frac{D(k)}{k}\right)^2+\ldots\right) = \frac1k - \frac{D(k)}{k^2} +\ldots
$$

where the remaining terms are of order $\frac{1}{M^3}$.
The expression in the summand of $P_M$ then becomes
$$ M\left(\frac{1}{k}-\frac{1}{k+D(k)}\right) \sim M\left(\frac{1}{k}-\left(\frac1k - \frac{D(k)}{k^2}\right)\right) = M\cdot \frac{D(k)}{k^2}$$

Set $x = \frac kM \in [\tfrac12,1]$. Then as $M \to \infty, D(k)\sim \frac{2x-1}{2}$. Then 
$$
M \left( \frac{1}{k} - \frac{1}{k + D(k)}\right) \sim M\cdot \frac{D(k)}{k^2} \sim M\frac{(2x-1)}{x^2M^2} = \frac{2x-1}{x^2}\cdot \frac 1M
$$

The sum is a Riemann sum, so as $M \to \infty,$
$$
P_M \to \int_{1/2}^{1} \frac{2x-1}{2x^2}\,dx.
$$
Evaluating this integral gives $$P_\infty^\text{HH} = P_\infty^\text{DEAN} = P_\infty^\text{WEB} = \ln2 - \frac12 \approx 0.1932. $$
\end{proof}
Summarizing the results of this section:
\begin{theorem}\label{Mtoinftyproblimit}
    Under the asymptotic uniform distribution for vectors $(1,x,y)$ with $1<x<y$, we have as $M \to \infty$ that the probability of a quota violation caused by nonzero allocations converges to $2 \ln(2) - 1 \approx 0.386294$ in the Adams method, to $0$ in the Modified Jefferson method, and to $\ln(2) - \frac{1}{2} \approx 0.1932$ in the Huntington-Hill, Modified Webster, and Dean methods.
\end{theorem}


\subsection{Probability of Quota Violations along $\tau$-Lines }

Next, for divisor methods $A$ corresponding to Huntington-Hill, modified Webster, modified Jefferson, or Dean, extend the results of the previous section to characterize when quota violations occur due to a guaranteed seat. 

To do this, the first step is to recall the threshold $x_\tau^F$ that quantifies when the asymptotic behavior in Theorem \ref{X_tau_def/exist} effectively holds. Specifically, for $x> x_\tau^F$, the guaranteed-seat state causes floor stabilization and potential quota violations. Equivalently, $x_{\tau}^F$ (from Theorem \ref{X_tau_def/exist}) is the infimum of all $z$ such that for all $x > z$, the following floor conditions hold (i.e., the floors stabilize).
\begin{itemize}
\item $\left\lfloor \frac{M(1-3\tau)}{(3-3\tau)x-3\tau} \right\rfloor = 0$

\item $\left\lfloor \frac{M(1-3\tau)x}{(3-3\tau)x-3\tau} \right\rfloor = \left\lfloor \frac{M(1-3\tau)}{3-3\tau} \right\rfloor$

\item $\left\lfloor \frac{2M(1-3\tau)x - M + 3M\tau}{(3-3\tau)(1-3\tau)x - 3\tau(1-3\tau)} \right\rfloor = \left\lfloor \frac{2M}{3-3\tau} \right\rfloor.$
\end{itemize}

By Theorem \ref{X_tau_def/exist}, for such $x$, the apportionment converges to $\tilde{A}(\tilde{q}_2, \tilde{q}_3)$, with $x_\tau$ being the infinum over all $z > x_\tau^F$ such that $\tilde{A}(\tilde{q}_2, \tilde{q}_3)=A(1,z,y(z,\tau))$. \\
\\
On a $\tau$-line, parametrized in $y(x,\tau)$
$$ y(x,\tau)= \frac{2x-1}{1-3\tau} = \frac{2}{1-3\tau}x - \frac{1}{1-3\tau} $$

Define then for a fixed $\tau$,
$$ y_\tau = y(x_\tau,\tau) \qquad\text{and}\qquad y_\tau^F = y(x_\tau^F,\tau)$$

Also recall along a $\tau$-line, the reduced population $(1,x,y)$ can be parametrized as $(1,x(y,\tau),y)$, where $x(y, \tau) = \frac{1}{2}y - \frac{3}{2}\tau y + \frac{1}{2},$ so that the total population is $1 + x(y,\tau) + y = \frac{3-3\tau}{2}y + \frac32$. Then the standard quotas, as functions of $y$, are: 
$$ q_1(y) = \frac{M}{1 + x(y,\tau) + y}, q_2(y) = \frac{Mx(y,\tau)}{1 + x(y,\tau)+y}, q_3(y) = \frac{My}{1 + x(y,\tau) + y}. $$
Next, determine $y_\tau^F$, the smallest value of $y$ such that the floor stabilization conditions hold along the $\tau$-line.

\begin{deflem}[Guaranteed-Seat Threshold]\label{y_tau.def}
Fix $\tau \in (-\tfrac{1}{3}, \tfrac{1}{3})$. There exists a constant $y_{\tau}^F > 0$ such that for all $y > y_{\tau}^F$, the floors of the standard quotas of the distribution $(1, x(y,\tau), y)$ stabilize:
$$\bigl(\lfloor q_1(y)\rfloor, \lfloor q_2(y)\rfloor, \lfloor q_3(y)\rfloor\bigr)
=\bigl(0, \lfloor \tilde q_2 \rfloor, \lfloor \tilde q_3 \rfloor \bigr),$$

where
$$\tilde q_2 = \frac{M(1-3\tau)}{3-3\tau}, \qquad \tilde q_3 = \frac{2M}{3-3\tau}.$$
\end{deflem}

\begin{proof}
Analyze each quota separately. For state 1,
$$ q_1(y) = \frac{M}{\frac{3-3\tau}{2}y + \frac{3}{2}}.$$
This is strictly decreasing in $y$ and satisfies $q_1(y) \to 0$ as $y \to \infty$. Then there exists $y_1(\tau)$ such that for all $y > y_1(\tau)$,
$0 < q_1(y) < 1$, so $\lfloor q_1(y) \rfloor = 0$.

For state 3, 
$$q_3(y) = \frac{My}{\frac{3-3\tau}{2}y + \frac{3}{2}}= \frac{2M}{3-3\tau + \frac{3}{y}}.$$
So $q_3(y)$ is strictly increasing and converges to $\tilde q_3$ as $y \to \infty$.

If $\tilde q_3 \in \mathbb{Z}$, then $q_3(y)$ converges monotonically to $\tilde q_3$, so there exists $y_3(\tau)$ such that $\lfloor q_3(y) \rfloor = \tilde q_3 - 1$ for all sufficiently large $y$. 

If $\tilde q_3 \notin \mathbb{Z}$, choose $\epsilon > 0$ small enough that
$$(\tilde q_3 - \varepsilon, \tilde q_3 + \varepsilon)
\subset \bigl(\lfloor \tilde q_3 \rfloor, \lfloor \tilde q_3 \rfloor + 1\bigr).$$

Then there exists $y_3(\tau)$ such that for all $y > y_3(\tau)$,
$|q_3(y) - \tilde q_3| < \epsilon,$ and as such, $\lfloor q_3(y) \rfloor = \lfloor \tilde q_3 \rfloor.$

For state 2, use $x(y,\tau)$:
$$
q_2(y) = \frac{Mx(y,\tau)}{\frac{3-3\tau}{2}y + \frac{3}{2}}
= \frac{M\left(\frac{1}{2}y - \frac{3}{2}\tau y + \frac{1}{2}\right)}{\frac{3-3\tau}{2}y + \frac{3}{2}}.
$$
A direct computation shows that $q'_2(y)$ has constant sign depending on $\tau$:
$$q_2'(y)
\begin{cases}
>0 & \text{if } \tau < 0,\\
=0 & \text{if } \tau = 0,\\
<0 & \text{if } \tau > 0,
\end{cases}
$$
so $q_2(y)$ is monotone and converges to $\tilde{q}_2$ as $y \to \infty$.

Thus, as for state 3, there exists $y_2(\tau)$ such that for all $y > y_2(\tau)$,
$\lfloor q_2(y) \rfloor = \lfloor \tilde q_2 \rfloor.$
Finally, define
$$y_{\tau}^F := \max\{y_1(\tau), y_2(\tau), y_3(\tau)\}.$$
Then for all $y > y_{\tau}^F$,
$$(\lfloor q_1(y)\rfloor, \lfloor q_2(y)\rfloor, \lfloor q_3(y)\rfloor)
=\left(0, \lfloor \tilde q_2 \rfloor, \lfloor \tilde q_3 \rfloor\right).$$
Additionally, a direct computation grants:
\begin{align*}
y_1(\tau) &= \frac{2M-3}{3-3\tau}\\
y_2(\tau) &= \left\{ \begin{array}{lcl} \frac{3 \lfloor \tilde{q}_2 \rfloor -M}{M-3M\tau-3\lfloor \tilde{q}_2 \rfloor+3\tau \lfloor \tilde{q}_2 \rfloor} & \tau<0 \\ 1 &\tau=0 \\ \frac{3 (\lfloor \tilde{q}_2 \rfloor+1) -M}{M-3M\tau-3(\lfloor \tilde{q}_2 \rfloor+1)+3\tau (\lfloor\tilde{q}_2 \rfloor+1)} & \tau >0  \end{array} \right. \\
y_3(\tau) &= \frac{3 \lfloor \tilde{q}_3 \rfloor}{2M-3\lfloor \tilde{q}_3 \rfloor +3\tau\lfloor \tilde{q}_3 \rfloor}.
\end{align*}

\end{proof}

Note by definition and construction, $y_{\tau}^F \leq y_\tau$ since $y_\tau$ imposes more conditions. For a fixed $\tau$, the constant $y_{\tau}^F$ identifies the smallest reduced population of the largest state along the $\tau$-line for which the floors of all standard quotas stabilize. When $y > y_{\tau}^F$, the guaranteed seat for state 1 could create an apportionment in such a way that state 3 receives fewer seats than its lower quota, producing a quota violation. This observation motivates the following lemma.

\begin{lemma}\label{y_tau.violation}
Let $\tau$ be ultimately violatory. If $y > y_{\tau}^F$, then the apportionment $\tilde A_1(1, x(y,\tau), y)$ produces a quota violation caused by the guaranteed seat of state 1. Equivalently, if $\tau$ is ultimately violatory $y_\tau^F = y_\tau$
\end{lemma}

\begin{proof}
For $y > y_{\tau}^F$, by Lemma \ref{y_tau.def}, 
$$(\lfloor q_1 \rfloor, \lfloor q_2 \rfloor, \lfloor q_3 \rfloor) = (0, \lfloor \tilde q_2 \rfloor, \lfloor \tilde q_3 \rfloor).$$

Since state 1 is guaranteed a seat, by Theorem \ref{36} the only possible apportionments are either $(1,\lfloor \tilde q_2 \rfloor,\lfloor \tilde q_3 \rfloor)$ or $(1,\lfloor \tilde q_2 \rfloor+1,\lfloor \tilde q_3 \rfloor-1)$. Since $\tau$ is ultimately violatory, the final apportionment must be $(1,\lfloor \tilde q_2 \rfloor+1,\lfloor \tilde q_3 \rfloor-1)$. This means that either the apportionment where $y > y_{\tau}^F$ is either always $(1,\lfloor \tilde q_2 \rfloor+1,\lfloor \tilde q_3 \rfloor-1)$, in which case, the Lemma is proved, or is the first apportionment and becomes $(1,\lfloor \tilde q_2 \rfloor+1,\lfloor \tilde q_3 \rfloor-1)$. However, this latter scenario violates population monotonicity on the third state. 
\end{proof}

\begin{deflem}[Non-Violatory Threshold]\label{y_tau.def2}
Fix $\tau \in (-\frac{1}{3},\frac{1}{3})$ that is ultimately non-violatory. 
There exists a constant $y_\tau > 0$ such that for all $y > y_\tau$, the apportionment $\tilde A_1(\tilde q_2, \tilde q_3)$ stabilizes without producing a quota violation. 

In particular, $y_\tau$ is given by
$$y_\tau := \max\{y_\tau^F, y_\tau^*\},$$
where $y_\tau$ is the threshold from Lemma \ref{y_tau.def}, and $y_\tau^*$ solves the priority-equality condition
$$
\frac{M y}{\delta(\lfloor \tilde q_3 \rfloor-1)} =
\frac{M\left(\frac{1}{2} y - \frac{3}{2} \tau y + \frac{1}{2}\right)}{\delta(M - \lfloor \tilde q_3 \rfloor - 1)}.
$$
\end{deflem}
\begin{proof}
First consider the case where $y_\tau^F$ does not yet calculate $y_\tau$. As in the previous proof, the only valid apportionments after $y_\tau^F$ are $(\lceil q_1 \rceil, \lceil q_2 \rceil, \lfloor q_3 \rfloor - 1)$ and $(\lceil q_1 \rceil, \lfloor q_2 \rfloor, \lfloor q_3 \rfloor)$.
Since by assumption the final apportionment is non-violatory, the apportionment must be moving from $(\lceil q_1 \rceil, \lceil q_2 \rceil, \lfloor q_3 \rfloor - 1)$ toward $(\lceil q_1 \rceil, \lfloor q_2 \rfloor, \lfloor q_3 \rfloor).$

Consider the apportionment functions $A(\tilde q_2, \tilde q_3)$ and $\tilde A_1(\tilde q_2, \tilde q_3)$. Before reaching $y_\tau$, state 3 receives its floor (and potentially loses a seat), while state 2 receives its ceiling. This implies that, at that point, the priority value for assigning state 3 an additional seat is less than the priority value for assigning state 2 an additional seat. 

As $y$ increases past $y_\tau$, the last seat is assigned in such a way that the priority values equalize, and the inequality flips. Solving for the value of $y$ where the priority values are equal gives
$$
\frac{M y}{\delta(\lfloor \tilde{q}_3 \rfloor-1)} = \frac{M\left(\frac{1}{2} y - \frac{3}{2} \tau y + \frac{1}{2}\right)}{ \delta(\lfloor \tilde q_2 \rfloor)}.$$

Solving this equation gives
$$y_\tau^* := \frac{\delta(\lfloor \tilde{q}_3 \rfloor-1)}{2\,\delta(\lfloor \tilde q_2 \rfloor) + (3\tau - 1)\,\delta(\lfloor \tilde{q}_3 \rfloor-1)}.$$

Finally, it follows the non-violatory threshold is then
$$y_\tau := \max\{y_\tau^F, y_\tau^*\}.$$
\end{proof}

\noindent
To summarize the above results, for $y \in (y_\tau^F, y_\tau)$, the guaranteed seat of state 1 forces one of the other states below its lower quota, producing a temporary quota violation. This violation disappears once $y > y_\tau$, consistent with the distribution being ultimately non-violatory. This gives the following:

\begin{corollary}\label{J_t_c}
For distributions that are ultimately non-violatory, if $y \in (y_\tau^F, \,y_\tau)$, then the guaranteed seat of state one temporarily causes a quota violation in the apportionment.
\end{corollary}

Additionally, the preceding results allow a characterization of the structure of $x_\tau$:
\begin{theorem}(Local structure of $x_\tau)$\label{for_remark}
The $\tau$-interval $(-\frac{1}{3}, \frac{1}{3})$ can be decomposed into  finitely many intervals such that on each interval, $x_\tau$ is locally represented in the form $x_\tau =\frac{a}{b+c\tau} + \frac{d\tau}{b+c\tau} + e$ for $a,b,c,d,e \in \mathbb{R}$ and denominators not identically $0$.
\end{theorem}
\begin{proof}
By Definition/Lemma \ref{X_tau_def/exist} and Lemma \ref{y_tau.violation}, $y_\tau$ is given by the maximum of functions of the form
\[ \frac{a}{b+c\tau} \]
when $\tau$ is ultimately violatory. For $\tau$ ultimately non violatory, the same form holds by Lemma \ref{y_tau.def2}. Decomposing into intervals for violatory and non-violatory, further decomposing based on which component is the maximum, and recalling $x(y,\tau) = \frac{1}{2}y - \frac{3}{2}\tau y + \frac{1}{2}$ grants the desired result.
    
\end{proof}

Next, look at $y$ values less than $y_\tau^F$ and determine the behavior of any possible quota violations. 

\begin{lemma}\label{y<T(t)}
If $y < y_{\tau}^F$, then no quota violations occur that are caused by the guaranteed seat of state 1 in Huntington-Hill, Modified Jefferson, Dean, and Modified Webster methods.
\end{lemma}
\begin{proof}
Let $A$ denote one of the Huntington-Hill, Modified Jefferson, Dean, or Modified Webster methods with divisor function $\delta$. Let the quotas satisfy $q_1 < q_2 < q_3$ and fix $\tau$. 

Recall $y_i(\tau)$ from Definition/Lemma \ref{y_tau.def}. We may restrict to the case $y>y_1(\tau)$ since if $y \leq y_1(\tau)$ $q_1 \geq 1$ and hence a quota violation caused by the nonzero allocation cannot occur.

The proof proceeds by considering two cases, $\tau \le 0$ and $\tau > 0$, due to monotonicity of the quota $q_2(y)$ depending on the sign of $\tau$.

\textbf{Case 1: $\tau \le 0$}. 
Suppose a quota violation exists. By Lemma \ref{36}, the floors of the quotas would have the form
$$
(\lfloor q_1 \rfloor, \lfloor q_2 \rfloor, \lfloor q_3 \rfloor) = (0, \lfloor \tilde q_2 \rfloor \pm s, \lfloor \tilde q_3 \rfloor \mp s), \quad s > 0.
$$
But for $\tau \le 0$, $q_1$ and $q_2$ are non-decreasing towards their limits $\tilde q_2, \tilde q_3$. Hence, such floors cannot occur, so no violation exists.

\textbf{Case 2: $\tau > 0$}. 
First, note that we may assume $y_{(2,+)} \leq \max(y_2, y_1)$, as otherwise there would exist $y$ such that the quotas at $y$ satisfy $(\lfloor q_1 \rfloor , \lfloor q_2 \rfloor, \lfloor q_3 \rfloor) = (0, \lfloor \tilde q_2 \rfloor +s, \lfloor q_3 \rfloor )$ where $s >0$. But then $\sum \lfloor q_i \rfloor \geq M$ and therefore there cannot be a quota violation by Theorem \ref{q.v_criteria_test}.

Given $y_1(\tau)<y < y_\tau^F$, it therefore suffices to check two intervals: $y_1 \le y \le y_{(2,+)}$ and $y_{(2,+)} \le y \le y_3$.

For $y_{(2,+)} \le y \le y_3$, the floor of $q_3$ is $\lfloor \tilde q_3 \rfloor - s$ with $s \ge 1$, so $\sum \lfloor q_i \rfloor < M-1$. By Theorem \ref{q.v_criteria_test}, no quota violation occurs.

For $y_1 \le y \le y_{(2,+)}$, the floor of $q_1$ is 0, and $q_2$ decreases monotonically towards $\tilde q_2$. At $y = y_1(\tau)$ compute, 
$$
q_2 = \tilde q_2 + \epsilon, \quad 0 < \epsilon < \frac{1}{2},
$$

It follows then for all $y \in (y_1(\tau), y_{(2,+)}(\tau))$ that $\lfloor q_2 \rfloor =\lfloor \tilde q_2 \rfloor + 1$ and $\lfloor q_1 \rfloor = 0$. Additionally, the corresponding decimal parts $(d_1, d_2)$ trace a line segment entirely below the line $y = \frac{1}{2} x$ in the unit square. Interpreting Theorem \ref{q.v_criteria_test} geometrically in the manner described in \parencite{Self} as a feasible region of $\{(d_1, d_2)\}$ for lower quota violations and considering the largest possible feasible region given the previously stated floors grants that lower quota violations can only occur within the triangle given by $(0, \delta(1)-1), (2-\delta(1), \delta(1)-1)$, and $(0,1)$. Since the trajectory of $(d_1,d_2)$ lies outside this region, no lower quota violation occurs for $y_1(\tau) < y < y_{(2,+)}.$ 

Combing the two intervals, we conclude for $\tau > 0$ no quota violations caused by nonzero allocations occur when $y < y_\tau^F$. Together with the case $\tau \leq 0$, the result follows. 
\end{proof}

Summarizing the above results gives the following:

\begin{theorem}[Quota Violation Probability]\label{kitchen_sink}
Let $M$ be the total number of seats, and let $A$ be the Modified Jefferson, Modified Webster, Dean, or Huntington-Hill method with divisor function $\delta$. Let $(1,x,y)$ be a reduced population vector with $1 < x < y$, drawn from a probability distribution with joint density $f(x,y)$.

For each $\tau \in \left(-\frac{1}{3},\frac{1}{3}\right)$, the $\tau$-line parametrization in $x$ is 
$$
x = x(y,\tau) := \frac{1}{2}y - \frac{3}{2}\tau y + \frac{1}{2},
$$
Let the set of ultimately violatory $\tau$ values be
$$
V = \bigcup_{k= \lfloor M/2 \rfloor + 1}^{M-1} 
\left( 1 - \frac{2M}{3k}, \, 1 - \frac{2M}{3\bigl(k + D(k)\bigr)} \right),
$$
where
$$
D(k) =\frac{(M-k)\delta(k-1)-k\,\delta(M-k-1)}{\delta(k-1)+\delta(M-k-1)}.
$$

For each $\tau$, let $y^F_\tau$ denote the floor-stabilization threshold and $y_\tau$ the agreement with asymptotic apportionment threshold defined in Lemma \ref{y_tau.def}. Define
$$
y_\tau^{\max} :=
\begin{cases}
\infty & \text{if } \tau \in V,\\
y_\tau & \text{if } \tau \notin V.
\end{cases}
$$

Then the probability of a quota violation caused by a guaranteed seat is
$$
\Pr(\text{quota violation})
=
\int_{-1/3}^{1/3}
\int_{y_\tau^F}^{y_\tau^{\max}}
f\bigl(x(y,\tau),\,y\bigr)\,
\frac{3}{2}y
\, dy \, d\tau.
$$
\end{theorem}

\begin{proof}
For a fixed $\tau \in (-\frac{1}{3},\frac{1}{3})$, the reduced population vector $(1,x,y)$ can be parametrized along the $\tau$-line by
$$
y = y(x,\tau) = \frac{2(x-1) + 3\tau(x+1)}{3-3\tau}.
$$
Along each such line, Lemmas~\ref{y_tau.def} and \ref{y_tau.def2} define the guaranteed-seat threshold $y_\tau^F$ and the agreement with asymptotic apportionment threshold $y_\tau$. 

By the preceding lemmas, a quota violation caused by the guaranteed seat of state 1 occurs precisely when $y > y_\tau^F$ and $y \le y_\tau^\mathrm{max}$, where
$$
y_\tau^\mathrm{max} = 
\begin{cases}
\infty, & \tau \in V \text{ (ultimately violatory)},\\
y_\tau, & \tau \notin V \text{ (temporarily violatory)}.
\end{cases}
$$
This follows from the monotonicity of the quotas $q_1(y),q_2(y),q_3(y)$ and the stabilization of their floors as $y$ increases. 

The probability of a quota violation is then obtained by integrating the probability density $f(x,y)$ over all $\tau$-lines and all $y$ in the violation interval:
$$
\Pr(\text{quota violation}) = \int_{-\frac{1}{3}}^{\frac{1}{3}} 
\int_{y^F_\tau}^{y_\tau^\mathrm{max}} f(x(y,\tau),y) \left|\frac{\partial x}{\partial y}\right| \, dy\, d\tau.
$$
Since the union of these intervals over all $\tau$ covers exactly the set of population vectors for which the guaranteed seat of state 1 causes a quota violation (temporarily or ultimately), the integral of $f(x,y)$ computes the total probability. 
\end{proof}

\begin{remark}
    Lemma \ref{y<T(t)} does not hold for the Adams method. A counterexample is the population vector $(1,6.25,115)$ which produces a quota violation caused by nonzero allocations, but $y = 115 < y_\tau^F = 570$. Consequently, the prior Theorem \ref{kitchen_sink} can, in general, provide only a lower bound on the probability of a quota violation caused by nonzero allocations when applied to the Adams method.\end{remark}

\subsection{Summary of notation used}
For clarity, the following is a list of relevant notation used in the paper.

\begin{table}[H]
\centering
\caption{Threshold notation summary for the $\tau$-stabilization framework.}
\label{tab:threshold-notation-compact}
\renewcommand{\arraystretch}{1.3}
\begin{tabular}{@{}cll@{}}
\toprule
\textbf{Symbol} & \textbf{Description} & \textbf{Defined in} \\
\midrule
\multicolumn{3}{@{}l}{\textit{Thresholds in $x$-parametrization:}} \\
$x_\tau^A$ & Apportionment equals limiting value $\tilde{A}(\tilde{q}_2, \tilde{q}_3)$ & Theorem \ref{X_tau_def/exist} \\
$x_\tau^F$ & Quota floors stabilize at $\lfloor \tilde{q}_i \rfloor$ & Theorem \ref{X_tau_def/exist} \\
$x_\tau$ & $\max\{x_\tau^A, x_\tau^F\}$: full stabilization & Theorem \ref{X_tau_def/exist} \\
\midrule
\multicolumn{3}{@{}l}{\textit{Thresholds in $y$-parametrization:}} \\
$y_\tau^F$ & Floor stabilization in $y$; equals $y(x_\tau^F, \tau)$ & Def./Lemma \ref{y_tau.def} \\
$y_\tau$ & $\max\{y_\tau^F, y_\tau^*\}$: non-violatory threshold & Def./Lemma \ref{y_tau.def2} \\
$y_\tau^{\max}$ & Upper integration limit ($\infty$ if $\tau \in V$, else $y_\tau$) & Theorem \ref{kitchen_sink} \\
\midrule
\multicolumn{3}{@{}l}{\textit{Limiting quotas and violation sets:}} \\
$\tilde{q}_2, \tilde{q}_3$ & $\lim_{x\to\infty} q_2(x), q_3(x)$ along $\tau$-line & Page 9 \\
$V$ & Set of ultimately violatory $\tau$ values & Theorem \ref{kitchen_sink} \\
$D(k)$ & Length of violation interval for $\lfloor \tilde{q}_3 \rfloor = k$ & Def./Lemma \ref{d.star} \\
\bottomrule
\end{tabular}
\end{table}


\section{Conclusion}
The results in this paper contribute both conceptually and practically to apportionment theory. Conceptually, $\tau$ provides a geometrically transparent coordinate that captures skewness and identifies families of distributions with uniform asymptotic behavior. Practically, the probability formulas and limits quantify how frequently common divisor methods violate quota when per‑state minimums are enforced, giving policymakers a principled basis for comparing methods. 

This paper is a natural continuation of the authors' first paper. In that paper, (\cite{Self}) quota violations that are intrinsic to the divisor method itself (violations that arise purely from the population distribution and the method's rounding behavior) are analyzed. This nonzero divisor paper studies a more constrained setting, isolating violations that are caused or worsened by the requirement that every state receive at least one seat. Definition \ref{vio_as_not0allo_def} makes this distinction precise. The combined results would allow one to compute, for each divisor method, the fraction of all quota violations in the constrained system that are directly attributable to the nonzero allocation constraint rather than to the method's inherent tendency to violate quota. This is a quantity of direct practical relevance to the design of apportionment systems with minimum representation guarantees.

This work extends a growing body of research on probabilistic approaches to apportionment in social choice theory and electoral mathematics. Most directly, it builds upon earlier work found in \cite{Self}. Additionally, it is related to analyses of systemic bias in apportionment methods \parencite{SchusterPukelsheimDrtonDraper2003, SchwingenschloglDrton2004, DrtonSchwingenschlogl2005, Ichimori2012}, and investigations into the likelihood of voting paradoxes \parencite{gehrlein2017elections, pandit-cutrone-fairness-2025}.

Natural next steps include refining the probabilistic model to reflect empirical population distributions (state‑size histograms or heavy‑tailed models), deriving finite‑$M$ error bounds for the Riemann‑sum approximations, computing explicit polyhedral volumes for small $n$ beyond three, and extending the analysis to combinations of constraints (e.g., regional minimum guarantees or multi-member districts). Applications range from legislative seat allocation and international treaty design to automated resource allocation systems that must satisfy fairness and minimal entitlement constraints. Further empirical and computational study will help translate the asymptotic insight provided here into concrete guidance for real-world apportionment decisions.

\section{Conflict of Interest Statement}
The authors declare that they have no conflict of interest.

\section {Funding Statement}
This research received no external funding.
\newpage
\appendix
\section*{Appendix}
\addcontentsline{toc}{section}{Appendix}
\section{Comparison of Theoretical Results to Simulations}\label{AppA}
The following are the results of simulations and comparisons to theoretical values calculated using Theorems \ref{prob_sum} and \ref{kitchen_sink}.

The python code used to sample quota violations is available on GitHub here:\\ \url{https://github.com/TylerCWunder/Probability-of-Quota-Violations-in-Divisor-Apportionment-Methods-with-Nonzero-Allocations.git}.

For the case where $(p_1, p_2, p_3) = (1,x,y)$ are uniform on $\{1<x<y\}$ we compare a sample of 100,000 $(p_1, p_2, p_3) = (1,x,y)$ where $x,y$ are both uniform random variables on $[1,1000000]$. By symmetry and as the range is large, one expects this sample to be comparable to the formula found in Theorem \ref{prob_sum}.
\begin{center}
\begin{tabular}{|c|c|c|c|c|}
    \hline
    \multicolumn{5}{|p{6.5in}|}{\centering \textbf{Probability of Quota Violations Caused by Nonzero Allocation\\ with $(1,x,y)$ uniform on $\{1<x<y\}$}}\\
    \hline
    
    Method & $M$ & Theoretical Probability & Sample Probability & $95\%$ Confidence Interval  \\ \hline
    Huntington-Hill & $10$ & $0.257$ &$0.257$ & $(0.254, 0.260)$ \\
    Huntington-Hill & $20$ & $0.229$ & $0.228$ & $(0.225, 0.230)$ \\
    Huntington-Hill & $50$ & $0.209$ & $0.209$ & $(0.207, 0.212)$ \\

    Adams & $10$ & $0.380$ & $0.380$ & $(0.376, 0.383) $\\
    Adams & $20$ & $0.385$ & $0.386$ & $(0.383, 0.389) $\\
    Adams & $50$ & $0.386$ & $0.385$ & $(0.382, 0.388)$ \\

    Dean & $10 $& $0.278 $& $0.280 $& $(0.277, 0.283)$ \\
    Dean & $20$ & $0.244 $& $0.244 $& $(0.242, 0.247)$ \\
    Dean & $50$ & $0.218$ & $0.218$ & $(0.215, 0.221)$ \\

    Modified Jefferson & $10$ & $0.111 $& $0.111$ & $(0.109, 0.113) $\\
    Modified Jefferson & $20$ & $0.053$ & $0.053$ & $(0.052, 0.054)$ \\
    Modified Jefferson & $50$ & $0.020$ & $0.021$ & $(0.020, 0.022)$ \\

    Modified Webster & $10 $&$ 0.235 $&$ 0.236 $&$ (0.234, 0.239) $\\
    Modified Webster & $20 $& $0.213 $& $0.213 $& $(0.210, 0.215) $\\
    Modified Webster & $50$ & $0.201$ & $0.202$ & $(0.199, 0.204)$ \\

    \hline
\end{tabular}
\end{center}

For the case $(p_1, p_2, p_3)$ are IID and $p_i \sim \exp(1)$, we compare a sample of 100,000 $(p_1, p_2, p_3)$ to a theoretical probability calculated using Theorem \ref{kitchen_sink}. Note that these are the same probabilities for $p_i \sim \exp(\lambda)$ for any $\lambda$ or $(p_1, p_2, p_3) \sim \mathrm{Dir}(1,1,1)$. 

\begin{center}
\begin{tabular}{|c|c|c|c|c|}
    \hline
    \multicolumn{5}{|p{6.5in}|}{\centering \textbf{Probability of Quota Violations Caused by Nonzero Allocation \\with $(p_1,p_2,p_3)$ IID and $p_i \sim \exp(1)$}}\\
    \hline
    
    Method & $M$ & Theoretical Probability & Sample Probability & $95\%$ Confidence Interval  \\ \hline
    Huntington-Hill &$5$ & $0.131$ &$0.130$ &$(0.128  , 0.132 )$ \\
    Huntington-Hill &$10$ & $0.049$ &$0.049$ &$(0.048, 0.051) $ \\
    Huntington-Hill &$15$ & $0.029$ &$0.029$ &$(0.028,0.030)$ \\

    Dean &$5$ & $0.137$ & $0.137$ & $(0.135, 0.139)$\\
    Dean &$10$ & $0.056$ & $0.056$ & $(0.055,0.057)$\\
    Dean &$15$ & $0.033$ & $0.033$ & $(0.032,0.034)$\\

    Modified Jefferson &$5$ &$0.120$ & $0.119$ &$(0.117, 0.121)$ \\
    Modified Jefferson &$10$ &$0.030$ & $0.030$ &$(0.029,0.031)$ \\
    Modified Jefferson &$15$ &$0.013$ & $0.013$ &$(0.012,0.014)$ \\

    Modified Webster & $5$ & $0.126$ & $0.125$ & $(0.123, 0.127)$\\
    Modified Webster & $10$ & $0.043$ & $0.043$ & $(0.042,0.045)$\\
    Modified Webster & $15$ & $0.025$ & $0.025$ & $(0.024,0.026)$\\

    \hline
\end{tabular}
\end{center}

\clearpage
\nocite{ElHelaly2019}
\nocite{FYS_Textbook}
\nocite{huntington-a-new}
\nocite{Huntington-theory}
\nocite{b-young-methods}
\printbibliography

\vspace{2in}
\begin{contact}
Joseph Cutrone\\
Department of Mathematics\\
Johns Hopkins University\\
Baltimore, Maryland, USA\\
\email{jcutron2@jhu.edu}\\\\

Tyler Wunder\\
Johns Hopkins University\\
Baltimore, Maryland, USA\\
\email{twunder2@jhu.edu}
\end{contact}


\end{document}